\documentclass{amsart}

\usepackage{amssymb}
\usepackage{bbm}
\usepackage{hyperref}

\newcommand*{\mailto}[1]{\href{mailto:#1}{\nolinkurl{#1}}}
\newcommand{\arxiv}[1]{\href{http://arxiv.org/abs/#1}{arXiv:#1}}
\newcommand{\doi}[1]{\href{http://dx.doi.org/#1}{doi: #1}}

\newtheorem{theorem}{Theorem}[section]

\newtheorem{lemma}[theorem]{Lemma}
\newtheorem{proposition}[theorem]{Proposition}
\newtheorem{corollary}[theorem]{Corollary}

\newcommand{\R}{{\mathbb R}}
\newcommand{\N}{{\mathbb N}}

\newcommand{\C}{{\mathbb C}}

\newcommand{\OO}{\mathcal{O}}
\newcommand{\oo}{o}

\newcommand{\be}{\begin{equation}}
\newcommand{\ee}{\end{equation}}

\newcommand{\spr}[2]{\langle #1 , #2 \rangle}

\newcommand{\E}{\mathrm{e}}

\newcommand{\sgn}{\mathrm{sgn}}
\newcommand{\tr}{\mathrm{tr}}

\newcommand{\dom}[1]{\mathrm{dom}(#1)}
\newcommand{\mul}[1]{\mathrm{mul}(#1)}
\newcommand{\ran}[1]{\mathrm{ran}(#1)}

\DeclareMathOperator{\linspan}{span}
\DeclareMathOperator{\supp}{supp}

\newcommand{\CR}{C_b(\R)}
\newcommand{\Lrom}{L^2(\R;|\om|)}
\newcommand{\Krom}{L^2(\R;\om)}

\newcommand{\Deftau}{\mathfrak{D}_\tau}

\newcommand{\M}{\mathcal{M}}

\newcommand{\BC}{BC}

\newcommand{\eps}{\varepsilon}

\newcommand{\sig}{\sigma}
\newcommand{\lam}{\lambda}

\newcommand{\om}{\omega}

\numberwithin{equation}{section}

\begin{document}

\title[On the isospectral problem of the Camassa--Holm equation]{On the isospectral problem of the dispersionless Camassa--Holm equation}

\author[J.\ Eckhardt]{Jonathan Eckhardt}
\address{Faculty of Mathematics\\ University of Vienna\\
Nordbergstrasse 15\\ 1090 Wien\\ Austria}
\email{\mailto{jonathan.eckhardt@univie.ac.at}}
\urladdr{\url{http://homepage.univie.ac.at/jonathan.eckhardt/}}

\author[G.\ Teschl]{Gerald Teschl}
\address{Faculty of Mathematics\\ University of Vienna\\
Nordbergstrasse 15\\ 1090 Wien\\ Austria\\ and International
Erwin Schr\"odinger
Institute for Mathematical Physics\\ Boltzmanngasse 9\\ 1090 Wien\\ Austria}
\email{\mailto{Gerald.Teschl@univie.ac.at}}
\urladdr{\url{http://www.mat.univie.ac.at/~gerald/}}

\thanks{Adv.\ Math.\ {\bf 235}, 469--495 (2013)}
\thanks{{\it Research supported by the Austrian Science Fund (FWF) under Grant No.\ Y330}}

\keywords{Camassa--Holm equation, isospectral problem, inverse spectral theory}
\subjclass[2010]{Primary 37K15, 34B40; Secondary 35Q35, 34L05}

\begin{abstract}
We discuss direct and inverse spectral theory for the isospectral problem of the dispersionless Camassa--Holm equation, where the weight is allowed to be a finite signed measure. In particular, we prove that this weight is uniquely determined by the spectral data and solve the inverse spectral problem for the class of measures which are sign definite. The results are applied to deduce several facts for the dispersionless Camassa--Holm equation. In particular, we show that initial conditions with integrable momentum asymptotically
split into a sum of peakons as conjectured by McKean.
\end{abstract}

\maketitle

\section{Introduction}

The Camassa--Holm equation
\begin{equation}\label{ch.eq}
u_{t}+2\varkappa u_{x}-u_{txx}+3uu_x=2u_x u_{xx}+u u_{xxx},\qquad x,\,t\in\R
\end{equation}
is an integrable, nonlinear wave equation which models unidirectional wave propagation on shallow water. Due to its many remarkable properties, this equation has attracted a lot of attention in the recent past and we only refer to \cite{beco}, \cite{const}, \cite{coes}, \cite{como}, \cite{cost} for further information. 
 In particular, more on the hydrodynamical relevance of this model can be found in the recent articles \cite{jo} and \cite{cla}.

The associated isospectral problem is given by the weighted Sturm--Liouville equation  
\begin{align}\label{eqnISOPdiffeqn}
 -f''(x) + \frac{1}{4} f(x) = z \om(x) f(x), \quad x\in\R,~z\in\C,
\end{align} 
 where the weight $\om$ is related to $u$ via $\om:=u-u_{xx} +\varkappa$. 
Direct, and in particular inverse, spectral theory for this Sturm--Liouville problem are of peculiar interest for solving the Cauchy problem of the Camassa--Holm equation and
will be the main focus of the present paper. More specific, we are solely interested in the dispersionless case where $\varkappa=0$ and $\om$ decays spatially.

Provided that the weight $\omega$ is a strictly positive function, it is well known that the spectral problem~\eqref{eqnISOPdiffeqn} (with possibly some suitable boundary conditions) gives rise to a self-adjoint operator in a weighted Hilbert space $L^2(\R;\omega)$. Moreover, if $\omega$ is smooth enough, it is even possible to transform this problem into a (in general singular) Sturm--Liouville problem in potential form and some inverse spectral conclusions may be drawn from this.
However, in order to incorporate the main interesting phenomena (wave breaking \cite{coes1}, \cite{coes2} and multi-peakon solutions \cite{bss}, \cite{cost}) 
of the dispersionless Camassa--Holm equation, it is necessary to treat the case where $\omega$ is a finite signed Borel measure on $\R$. In fact, multi-peakon solutions of the Camassa--Holm equation correspond to weights  which are a finite sum of (real-valued) weighted Dirac measures and wave breaking only occurs if the weight changes sign.
Therefore, we will investigate \eqref{eqnISOPdiffeqn} when the weight $\omega$ is assumed to be an arbitrary finite signed Borel measure on $\R$. 

As noted by Beals, Sattinger, and Szmigielski \cite{bss}, the above isospectral problem~\eqref{eqnISOPdiffeqn} is closely related to the spectral problem for an indefinite string 
\begin{align}\label{eqnString}
 -g''(y) = z\, m(y) g(y), \quad y\in(-1,1),~z\in\C,
\end{align} 
 with a signed Borel measure $m$ on $(-1,1)$. 
 In fact, \eqref{eqnString} is equivalent to \eqref{eqnISOPdiffeqn} by virtue of the Liouville transform  
 \begin{align*}
   y & =\tanh\left(\frac{x}{2}\right), & g(y) & = f(x)\cosh^{-1}\left(\frac{x}{2}\right), &  m(y) & = 4\cosh^4\left(\frac{x}{2}\right) \om(x).
 \end{align*}
 Hereby note that the class of finite signed measures $\om$ on $\R$ corresponds to the class of signed measures $m$ on $(-1,1)$ subject to the condition 
 \begin{align}\label{eqnCondString}
  \int_{-1}^1 (1-y^2) d|m|(y) < \infty. 
 \end{align}
 In particular, some measure $m$ corresponding to a finite signed measure $\om$ on $\R$ is not necessarily finite. 
 More precisely, the total variation of $m$ is finite near the endpoint $\pm 1$ if and only if $\om$ satisfies the growth condition
\begin{align}\label{eqnIntLCCond}
  \int_\R \E^{\pm x} d|\om|(x) < \infty,
\end{align}
 and equivalently, if and only if \eqref{eqnISOPdiffeqn}  is in the limit-circle case (regarded as a spectral problem in the weighted Hilbert space $\Lrom$) at $\pm\infty$. 

 For the case of a positive measure $m$ with one endpoint being regular, spectral theory for~\eqref{eqnString} is quite well developed (see \cite[\S 6]{ka} for two singular endpoints). 
 The corresponding inverse spectral problem (with a variable right endpoint) was solved by Kac and Krein \cite{kakr} (see also \cite{dymmck}, \cite{kac}, \cite{kotwat}). 
  In particular, they showed that the case where $m$ (and hence $\om$) is a finite sum of weighted Dirac measures can be explicitly solved by virtue of the Stieltjes moment problem.   
 We will complement these classical inverse results in Section~\ref{secIP}  by allowing both endpoints in~\eqref{eqnString} to be quite singular.
  Although we will state our theorems in the terminology of the isospectral problem \eqref{eqnISOPdiffeqn}, we obtain a complete and concise characterization of all spectral measures arising from~\eqref{eqnString} with Borel measures $m$ on $(-1,1)$ of one sign and satisfying~\eqref{eqnCondString}.

Spectral theory for \eqref{eqnString} (and hence also for~\eqref{eqnISOPdiffeqn}) in the indefinite case is more complicated and literature regarding the inverse problem is quite scarce (for example, see \cite{krla79}, \cite{krla80}, \cite{lanwin}). 
  Some of the difficulties arising there are already visible when $\omega$ is just a finite sum of real-valued weighted Dirac measures. 
  In this case, Beals, Sattinger, and Szmigielski \cite{bss} were able to solve the inverse problem employing the connection with the Stieltjes moment problem mentioned above.  
  However, the description of the set of spectral measures is more complicated than in the definite case.  
 Apart from this, to the best of our knowledge, the only inverse results which deal with~\eqref{eqnISOPdiffeqn} in the indefinite setting are due to Bennewitz \cite{ben3} and Bennewitz, Brown, and Weikard \cite{bebrwe}, who proved uniqueness of the inverse problem under several restrictions on the weight measure. 
 In particular, these restrictions exclude the case of two singular endpoints (in the sense that~\eqref{eqnIntLCCond} does not hold) and the case when the weight measure has gaps in its support (which excludes the multi-peakon case).   
The main contribution of the present paper to the indefinite inverse problem is an extension of their result to arbitrary finite signed measures $\omega$ on $\R$. 
This is done by employing an inverse uniqueness theorem from \cite{LeftDefiniteSL} for general left-definite Sturm--Liouville problems with measure coefficients.  
In particular, this approach is based on singular Weyl--Titchmarsh--Kodaira theory which was first introduced by Kodaira \cite{ko} for one-dimensional Schr\"odinger operators.
Some gaps were later pointed out by Kac \cite{ka} and fixed using an alternate approach. 
  However, these results did not get much attention until recently when Gesztesy and Zinchenko \cite{gz} took it up again and triggered a large amount of results \cite{LeftDefiniteSL}, \cite{et2}, \cite{ful08}, \cite{fl}, \cite{kt}, \cite{kt2}, \cite{kst}, \cite{kst2}, \cite{kst3}. 
  In particular, Kostenko, Sakhnovich, and Teschl \cite{kst2} were the first to prove an inverse spectral result, by showing that the singular Weyl function uniquely determines the potential. 
  That the same is true for the spectral measure was shown by Eckhardt \cite{je} using de Branges theory and then extended to left-definite Sturm--Liouville operators in \cite{LeftDefiniteSL}.

We will apply these inverse spectral results to investigate the Cauchy problem for the dispersionless Camassa--Holm equation. In fact, our 
main motivation to study this inverse spectral transform stems from the quest for long-time asymptotics. 
 It was generally believed, and for example conjectured by McKean in~\cite{mckean} (cf.\ also \cite{mckean2}, \cite{mckean3}), that solutions of the dispersionless Camassa--Holm equation will asymptotically  split into a train of well separated single peakons, each of which corresponding to an eigenvalue of the underlying isospectral problem. 
Apart from the multi-peakon case \cite{bss} (and some low-regularity solutions \cite{li} as well as simplified case \cite{lou}), this has still been an open question to the best of our knowledge. 
 The difficulty hereby stems from the fact, that in contrast to other nonlinear wave equations, the inverse problem cannot be reformulated as a proper
Riemann--Hilbert problem because of the absence of a continuous spectrum. This is only possible for the case with dispersion, that is $\varkappa>0$,
where one can develop an inverse scattering approach similar to the one for
the Korteweg--de Vries equation (see \cite{bebrwe2}, \cite{const}, \cite{cgi}). The resulting oscillatory Riemann--Hilbert problem can then be analyzed using the nonlinear
steepest decent analysis originally introduced by Manakov, Its, Deift, and Zhou (see the survey \cite{diz} or the expository introduction \cite{grte}).
For the Camassa--Holm equation with dispersion, this was done by Boutet de Monvel and Shepelsky \cite{bosh} (there the limit $\varkappa\downarrow 0$ was also considered),
Boutet de Monvel, Its, and Shepelsky \cite{boitsh} and Boutet de Monvel, Kostenko, Shepelsky, and Teschl \cite{bokoshte}. 
 In contradistinction, our approach relies on a comprehensive study of direct and inverse spectral theory for the isospectral problem.
  In particular, we derive some kind of continuity for the inverse spectral problem which then allows us to deduce long-time asymptotics using the well-known time evolution of the spectral quantities.

 It remains to give a short outline of our paper. 
 Section~\ref{secISO} introduces the iso\-spec\-tral operator in the weighted Hilbert space $\Lrom$ associated with the weighted Sturm--Liouville problem~\eqref{eqnISOPdiffeqn} and discusses its basic properties. 
 Subsequently, in the following two sections, we describe the spectral quantities associated with the isospectral operator, determine its resolvent and derive some growth restrictions for certain (spatially decaying) solutions of~\eqref{eqnISOPdiffeqn}. 
 In order to apply the inverse uniqueness theorem from \cite{LeftDefiniteSL}, in Section~\ref{secLD} we introduce the so-called left-definite operator in $H^1(\R)$, associated with~\eqref{eqnISOPdiffeqn}, and describe how it is related to the isospectral operator.
 The following section discusses some continuity properties of the one-to-one correspondence between weight measures and spectral measures. 
 In Section~\ref{secIP} we will solve the inverse spectral problem in the case of sign definite weight measures. 
 The results obtained for the isospectral problem are then applied in Section~\ref{secCH} to deduce several facts (of the type of \cite{grhora} and \cite{holray})  for the dispersionless Camassa--Holm equation. 
 Finally, in the last section we will derive long-time asymptotics for solutions of the dispersionless Camassa--Holm equation.

\section{The isospectral operator}\label{secISO}

Let $\omega$ be some arbitrary finite signed Borel measure on $\R$. 
The maximal domain $\Deftau$ of functions for which~\eqref{eqnISOPdiffeqn} makes sense consists of all locally absolutely continuous functions $f$ for which the function
\begin{align}\label{eqnQD}
 -f'(x) + f'(c) + \frac{1}{4} \int_c^x f(s)ds, \quad x\in\R
\end{align}
is locally absolutely continuous with respect to the measure $\omega$. Note that this is the case if and only if there is some $f_\tau\in L^1_{\mathrm{loc}}(\R;|\omega|)$ such that
\begin{align*}
 -f'(x) + f'(c) + \frac{1}{4} \int_c^x f(s)ds = \int_c^x f_\tau(s)d\omega(s)
\end{align*}
for almost all $x\in\R$ with respect to the Lebesgue measure. Here and henceforth, integrals with respect to the measure $\omega $ have to be read as
\begin{align*}
 \int_c^x f_\tau(s)d\omega(s) = \begin{cases}
                                     \int_{[c,x)} f_\tau(s) d\omega(s), & x>c, \\
                                     0,                                     & x=c, \\
                                     -\int_{[x,c)} f_\tau(s) d\omega(s), & x< c. 
                                    \end{cases}
\end{align*}
In other words, the function $f_\tau$ is the Radon--Nikod\'{y}m derivative of the function~\eqref{eqnQD} with respect to the measure $\omega$. In this case, that is, for functions $f\in\Deftau$, we set $\tau f = f_\tau$ and hence our differential equation~\eqref{eqnISOPdiffeqn} becomes $(\tau-z)f=0$. 
 Moreover, given some functions $f$, $g\in\Deftau$,  a simple integration by parts yields the equation  
\begin{align}\label{eqnintpar}
 \int_\alpha^\beta  f(x) \tau g(x)  d\omega(x) =  \frac{1}{4} \int_\alpha^\beta f(x)g(x)dx + \int_\alpha^\beta f'(x)g'(x) dx -\left[fg'\right]_{\alpha}^\beta 
\end{align} 
for each $\alpha$, $\beta\in\R$ with $\alpha<\beta$, which will be used repeatedly.

Associated with this differential expression $\tau$ is a linear operator $T$ in the weighted Hilbert space $\Lrom$, equipped with the inner product 
\begin{align*}
 \spr{f}{g}_{\Lrom} = \int_\R f(x) g(x)^\ast d|\omega|(x), \quad f,\, g\in\Lrom.
\end{align*}
In order to define it, we say a function $f\in\Deftau$ with $f$, $\tau f\in\Lrom$ satisfies the boundary condition at $\pm\infty$ if
\begin{align}\label{eqnBC}
 \BC_{\pm\infty}(f) := \lim_{x\rightarrow\pm\infty} \left(f(x) \pm \frac{1}{2} f'(x)\right)\E^{\mp\frac{x}{2}} = 0.
\end{align}
Hereby note that these limits are known to exist; see~\cite[Lemma~4.2]{measureSL}. 
 Moreover, the boundary condition at $\pm\infty$ is actually superfluous, provided that $\omega$ does not decay to fast near $\pm\infty$. In fact, it is automatically satisfied if and only if 
\begin{align*}
 \int_\R \E^{\pm x} d|\omega|(x) = \infty.
\end{align*}
This difference stems from the fact that $\tau$ is either in the limit-circle or in the limit-point case at each endpoint. 
Now the linear operator $T$ is given by
\begin{align*}
 \dom{T} = \lbrace f\in\Deftau \,|\, f,\,\tau f\in \Lrom,~ \BC_{+\infty}(f) = \BC_{-\infty}(f) = 0 \rbrace
\end{align*}
and $T f = \tau f$ for $f\in\dom{T}$. Hereby note that each function in $\dom{T}$ has a unique representative $f\in\Deftau$ with $f$, $\tau f\in\Lrom$ and $\BC_{+\infty}(f) = \BC_{-\infty}(f) =0$ in view of~\cite[Section~7]{measureSL}. Hence the differential operator $T$ is well-defined. 

Unfortunately this operator $T$ is not self-adjoint unless the measure $\omega$ is of one sign. However, it will turn out to be self-adjoint in the Krein space $\Krom$ (which coincides with $\Lrom$ as a set), equipped with the (in general) indefinite inner product
\begin{align*}
  \spr{f}{g}_{\Krom} = \int_\R f(x) g(x)^\ast d\omega(x), \quad f,\, g\in \Lrom.
\end{align*}
The Krein space topology of $L^2(\R;\omega)$ is precisely the Hilbert space topology of $\Lrom$, which makes the indefinite inner product continuous.  More precisely, the unitary operator $J$ of multiplication with the sign of $\omega$ in $\Lrom$, i.e.
\begin{align*}
 J f(x) = \frac{d\omega}{d|\omega|}(x)\, f(x), \quad x\in\R,~f\in\Lrom,
\end{align*}
 is a fundamental symmetry with 
 \begin{align*}
  \spr{f}{g}_{\Krom} = \spr{J f}{g}_{\Lrom}, \quad f,\, g\in\Krom.
 \end{align*}
 Furthermore, $JT$ turns out to be the differential operator associated with the measure $|\omega|$. From the results in~\cite[Section~6]{measureSL} we infer that the operator $JT$ is self-adjoint in $\Lrom$ and hence $T$ is self-adjoint in $\Krom$.  
 In particular, this guarantees that $T$ is a closed operator in $\Lrom$. 

Finally, we want to emphasize that we only assumed $\omega$ to be a finite signed measure. In particular, $\omega$ is allowed to vanish on arbitrary sets and may even be equal to zero at all. Of course, this last case leads to a degenerate Hilbert space $\Lrom$ and a degenerate operator $T$. Moreover, also the case when $\omega$ is supported on only one point is still quite degenerate, since in this case all solutions of~\eqref{eqnISOPdiffeqn} are linearly dependent in $\Lrom$. Nevertheless, it still gives rise to a (one-dimensional) self-adjoint linear operator $T$ in $\Lrom$; see \cite[Appendix~C]{measureSL}.

\section{Spectrum and resolvent}

As a first step we construct solutions of $(\tau-z)u=0$ which satisfy the boundary condition at $\pm\infty$ and depend analytically on $z\in\C$.

\begin{theorem}\label{thmSpecResPHI}
 For each $z\in\C$ there is a unique solution $\phi_\pm(z,\cdot\,)$ of the differential equation $(\tau -z)u=0$ with the spatial asymptotics
\begin{align}\label{eqnSpatAsy}
 \phi_\pm(z,x) \sim \E^{\mp\frac{x}{2}} \quad\text{and}\quad \phi_\pm'(z,x) \sim \mp \frac{1}{2} \E^{\mp\frac{x}{2}},
\end{align}
as $x\rightarrow\pm\infty$. Moreover, the functions
\begin{align}\label{eqnPHIEntire}
 z\mapsto \phi_\pm(z,x) \quad\text{and}\quad z\mapsto\phi_\pm'(z,x)
 \end{align}
 are real entire and of finite exponential type for each $x\in\R$.
\end{theorem}

\begin{proof}
 First of all we show that for each $z\in\C$, the integral equation
\begin{align}\label{eqnIsoInt}
 m_\pm(z,x) = 1 \pm z \int_{x}^{\pm\infty} \left(\E^{\pm(x-s)}-1\right) m_\pm(z,s)d\om(s), \quad x\in\R,
\end{align}
has a unique bounded continuous solution $m_\pm(z,\cdot\,)$.
To this end, consider the integral operator $K_\pm$ on $\CR$ 
\begin{align}\label{eqnIntOp}
 K_\pm f(x) = \pm \int_{x}^{\pm\infty} \left(\E^{\pm(x-s)}-1\right)f(s)d\om(s), \quad x\in\R,~f\in\CR,
\end{align}
where $\CR$ is the space of bounded continuous functions on $\R$. Note that for each $f\in\CR$ the function $K_\pm f$ is continuous since the integrand is bounded and continuous. 
Moreover, for each $n\in\N$ we have the estimate
\begin{align*}
 \sup_{s<x} \left|K_-^n f(s)\right| \leq \frac{1}{n!} \left(\int_{-\infty}^x d|\omega|\right)^n \sup_{s<x} \left|f(s)\right|, \quad x\in\R.
\end{align*}
In fact, the case when $n=1$ is easily verified. Otherwise we get inductively
\begin{align*}
 \sup_{s<x} \left|K_-^n f(s)\right| & \leq \sup_{s<x} \int_{-\infty}^s \left|\E^{r-s}-1\right| \left|K_-^{n-1}f(r)\right| d|\om|(r) \\
  & \leq \frac{1}{(n-1)!} \int_{-\infty}^x  \left(\int_{-\infty}^r d|\om|\right)^{n-1} d|\om|(r)\, \sup_{s<x} \left|f(s)\right|, \quad x\in\R.
\end{align*}
Now an application of the substitution rule for Lebesgue--Stieltjes integrals~\cite{tsub} yields the claim. Similarly one obtains a corresponding estimate for $K_+ f$ which ensures that $K_\pm f\in\CR$. Moreover, we even get the bound 
\begin{align}\tag{$*$}\label{eqnKneu}
 \|K_\pm^n\| \leq \frac{1}{n!} \left(\int_\R d|\omega|\right)^n, \quad n\in\N,
\end{align}
and hence the Neumann series 
\begin{align}\label{eqnNeumannseries}
 m_\pm(z,x) = \sum_{n=0}^\infty z^n K_\pm^n 1(x) = (I-zK_\pm)^{-1} 1(x), \quad x\in\R,~z\in\C,
\end{align}
converges absolutely, uniformly in $x\in\R$ and even locally uniformly in $z\in\C$.
In particular, this function $m_\pm(z,\cdot\,)$ is the unique solution in $\CR$ of the integral equation in~\eqref{eqnIsoInt}. Moreover, integrating the right-hand side of~\eqref{eqnIsoInt} by parts shows that this function is locally absolutely continuous with derivative given by
\begin{align}\label{eqnIntEqnDer}
 m_\pm'(z,x) =  z\int_{x}^{\pm\infty} \E^{\pm(x-s)} m_\pm(z,s)d\om(s), \quad x\in\R,~z\in\C.
\end{align}
Therefore, we have the spatial asymptotics 
\begin{align*}
 m_\pm(z,x) \rightarrow 1 \quad\text{and}\quad m_\pm'(z,x)\rightarrow 0,
\end{align*}
as $x\rightarrow\pm\infty$ for each $z\in\C$. Indeed, this follows from the integral equation (and its spatial derivative) and the fact that the function $m_\pm(z,\cdot\,)$ is uniformly bounded. 
Now equation~\eqref{eqnIsoInt} shows that the functions 
\begin{align*}
\phi_\pm(z,x) = \E^{\mp\frac{x}{2}}m_\pm(z,x), \quad x\in\R,~z\in\C,
\end{align*}
 satisfy the integral equations 
\begin{align*}
 \phi_\pm(z,x) & = \E^{\mp\frac{x}{2}} \pm z \int_{x}^{\pm\infty} \left(\E^{\pm\frac{x-s}{2}}-\E^{\mp\frac{x-s}{2}}\right) \phi_\pm(z,s)d\om(s),\quad x\in\R,~z\in\C.
\end{align*}
From this it is easily verified that $\phi_\pm(z,\cdot\,)$ is a solution of $(\tau-z)u=0$ (see e.g.~\cite[Proposition~3.3]{measureSL}).
The spatial asymptotics of $\phi_\pm(z,\cdot\,)$ near $\pm\infty$ easily follow from the corresponding results for the function $m_\pm(z,\cdot\,)$. 
Also note that these asymptotics uniquely determine the solution $\phi_\pm(z,\cdot\,)$. 
Finally, the Neumann series and the estimates in~\eqref{eqnKneu} guarantee that $m_\pm(\,\cdot\,,x)$ is real entire and of finite exponential type, uniformly for all $x\in\R$. Hence we see from~\eqref{eqnIntEqnDer} that $m_\pm'(\,\cdot\,,x)$ is also real entire with finite exponential type for each $x\in\R$. Of course, this proves that the functions in~\eqref{eqnPHIEntire} are real entire and of finite exponential type for each $x\in\R$. 
\end{proof}

From the spatial asymptotics of the solutions $\phi_\pm(z,\cdot\,)$, $z\in\C$ it is easily seen that they are square integrable with respect to $|\omega|$ near $\pm\infty$ and satisfy the boundary condition~\eqref{eqnBC} there.
In particular, this guarantees that the spectrum of $T$ is purely discrete and simple.
 More precisely, from~\cite[Theorem~8.5]{measureSL} and~\cite[Theorem~9.6]{measureSL} it follows that $JT$ and hence also $T$ has purely discrete spectrum. The fact that the spectrum is simple follows literally as in the proof of~\cite[Corollary~8.4]{measureSL}. 
 Consequently, some $\lambda\in\C$ is an eigenvalue of $T$ if and only if the solutions $\phi_-(\lambda,\cdot\,)$ and $\phi_+(\lambda,\cdot\,)$ are linearly dependent, that is, their Wronskian 
\begin{align*}
 W(z) = \phi_+(z,x)\phi_-'(z,x) - \phi_+'(z,x)\phi_-(z,x), \quad z\in\C
\end{align*}
vanishes in $\lambda$. 
 In this case there is a constant $c_{\lambda,\pm}\in\C^\times$ such that 
\begin{align}\label{eqnCoup}
 \phi_\pm(\lambda,x) = c_{\lambda,\pm} \phi_\mp(\lambda,x), \quad x\in\R.
\end{align}
 Moreover, the quantity 
\begin{align}\label{eqnNorm}
 \gamma_{\lambda,\pm}^{2} = \int_\R |\phi_\pm(\lambda,x)|^2 d\omega(x),
\end{align}
 is finite and referred to as the left/right norming constant associated with the eigenvalue $\lambda$.
 Using~\eqref{eqnintpar} and the spatial asymptotics of the solution $\phi_\pm(\lambda,\cdot\,)$ in combination with $(\tau-z)\phi_\pm=0$ shows that
\begin{align}\label{eqnNormConstH1}
 \lambda\gamma_{\lambda,\pm}^2 = \frac{1}{4} \int_\R |\phi_\pm(\lambda,x)|^2 dx + \int_\R |\phi_\pm'(\lambda,x)|^2 dx > 0.
\end{align}
 In particular, this guarantees that $\lambda$ and  hence the spectrum $\sigma(T)$ of $T$ is real. Moreover, from this equation one also sees that the spectrum is positive (respectively negative), provided that the measure $\omega$ is positive (respectively negative). The following lemma relates all these spectral quantities.

\begin{lemma}\label{lemWderlam}
 For each eigenvalue $\lambda\in\sigma(T)$ we have
\begin{align}\label{eqnWlam}
 - \dot{W}(\lambda) = \int_\R \phi_-(\lambda,x)\phi_+(\lambda,x)d\om(x) = c_{\lambda,\mp} \gamma_{\lambda,\pm}^{2} \not=0,
\end{align}
where the dot denotes differentiation with respect to the spectral parameter. 
\end{lemma}

\begin{proof}
We set
\begin{align}\tag{$*$}\label{eqnWpm}
 W_\pm(z,x) = \dot{\phi}_\pm(z,x) \phi_\mp'(z,x) - \dot{\phi}'_\pm(z,x)\phi_\mp(z,x), \quad x\in\R,~z\in\C,  
\end{align}
where the spatial differentiation is done first. Now, using the differential equation for the solution $\phi_\pm(z,\cdot\,)$ one gets
\begin{align*}
 W_\pm(z,\beta) - W_\pm(z,\alpha) = \int_\alpha^\beta \phi_-(z,s) \phi_+(z,s)d\omega(s), \quad \alpha,\,\beta\in\R.
\end{align*}
More precisely, this follows by differentiating~\eqref{eqnWpm} with respect to the spatial variable, where the derivative is in general a Borel measure.
Now differentiating the integral equation in~\eqref{eqnIsoInt} and its spatial derivative in~\eqref{eqnIntEqnDer} with respect to the spectral variable we get
\begin{align*}
 (I-zK_\pm) \dot{m}_\pm(z,\cdot\,) = K_\pm m_\pm(z,\cdot\,), \quad z\in\C,
\end{align*}
as well as
\begin{align*}
 \dot{m}_\pm'(z,x) = \int_x^{\pm\infty} \E^{\pm(x-s)} \left(m_\pm(z,s)+z\dot{m}_\pm(z,s)\right) d\omega(s), \quad x\in\R,~z\in\C.
\end{align*}
In particular this shows that
\begin{align*}
 \dot{m}_\pm(z,x) \rightarrow 0 \quad\text{and}\quad \dot{m}_\pm'(z,x) \rightarrow 0,
\end{align*}
as $x\rightarrow\pm\infty$ for each $z\in\C$. If $\lambda\in\sigma(T)$ is an eigenvalue, then we furthermore know that
\begin{align*}
 m_\mp(\lambda,x) = \E^{\mp\frac{x}{2}} \phi_\mp(\lambda,x) = \E^{\mp\frac{x}{2}} c_{\lambda,\mp} \phi_\pm(\lambda,x), \quad x\in\R,
\end{align*}
and hence $m_\mp(\lambda,x)$ and $m_\mp'(\lambda,x)$ are bounded as $x\rightarrow\pm\infty$. A calculation shows that for each $x\in\R$ we have 
\begin{align*}
 W_\pm(\lambda,x) = \dot{m}_\pm(\lambda,x) m_\mp(\lambda,x) + \dot{m}_\pm(\lambda,x) m_\mp'(\lambda,x) - \dot{m}_\pm'(\lambda,x) m_\mp(\lambda,x),
\end{align*}
which tends to zero as $x\rightarrow\pm\infty$. Therefore we conclude
\begin{align*}
 W_\pm(\lambda,x) = - \int_x^{\pm\infty} \phi_-(\lambda,s) \phi_+(\lambda,s) d\omega(s), \quad x\in\R,
\end{align*}
and hence finally
\begin{align*}
 - \dot{W}(\lambda) = W_-(\lambda,x) - W_+(\lambda,x) =  \int_\R \phi_-(\lambda,s)\phi_+(\lambda,s)d\omega(s), \quad x\in\R,
\end{align*}
which is the claimed identity.
\end{proof}

We will now determine the inverse of our operator $T$. 
 Therefore note that from the Neumann series~\eqref{eqnNeumannseries} we get
\begin{align*}
 \phi_\pm(0,x)= \E^{\mp\frac{x}{2}}, \quad x\in\R.
\end{align*}
Furthermore, this series yields an expansion of $\phi_\pm(\,\cdot\,,x)$ near zero which will be needed later on.

\begin{proposition}\label{propInverse}
The operator $T$ is invertible with inverse given by
\begin{align}
 T^{-1} g(x) = \int_\R \E^{-\frac{|x-s|}{2}} g(s)d\omega(s), \quad x\in\R, ~ g\in\Lrom.
\end{align}
Moreover, this inverse is a trace class operator with
\begin{align}\label{eqnTF}
  \sum_{\lambda\in\sigma(T)} \frac{1}{\lambda} = \int_\R d\omega \quad\text{and}\quad \sum_{\lambda\in\sigma(T)} \frac{1}{|\lambda|} \leq \int_\R d|\omega|,
\end{align}
where the inequality is strict if and only if $\omega$ changes sign.
\end{proposition}

\begin{proof}
 Since the solutions $\phi_-(0,\cdot\,)$ and $\phi_+(0,\cdot\,)$ are linearly independent, $JT$ is invertible with
 \begin{align*}
  (JT)^{-1} g(x) = \int_\R \E^{-\frac{|x-s|}{2}} g(s) d|\omega|(s), \quad x\in\R,~ g\in\Lrom,
 \end{align*}
 in view of~\cite[Theorem~8.3]{measureSL}. Thus $T$ is invertible as well with inverse given as in the claim. 
  Moreover, since the spectrum of $JT$ is positive in view of equation~\eqref{eqnNormConstH1}, $(JT)^{-1}$ is positive as well and we infer from the lemma on page~65
  in \cite[Section~XI.4]{RS2} that $(JT)^{-1}$ is even a trace class operator with trace norm
 \begin{align*}
   \tr\; (JT)^{-1} = \int_\R d|\omega|.
 \end{align*}
 But this shows that $T^{-1}$ is also a trace class operator with the same trace norm as $(JT)^{-1}$, which proves the inequality in~\eqref{eqnTF} in view of Weyl's majorant theorem (see ,e.g., \cite[Theorem~II.3.1]{gokr}). Moreover, this inequality is strict if and only if $T^{-1}$ is normal (and hence self-adjoint since its spectrum is real). 
 In order to compute the trace of $T^{-1}$, consider the positive integral operators
 \begin{align*}
  R_\pm g(x) = \int_\R \E^{-\frac{|x-s|}{2}} g(s) d\omega_\pm(s), \quad x\in\R,~ g\in L^2(\R;\omega_\pm),
 \end{align*}
 in the Hilbert spaces $L^2(\R;\omega_\pm)$, where $\omega=\omega_+ - \omega_-$ is the Hahn--Jordan decomposition of $\omega$.
 Now, if we identify $\Lrom$ with the orthogonal sum of the spaces $L^2(\R;\omega_+)$ and $L^2(\R;\omega_-)$, then we get 
 \begin{align*}
  \tr\; T^{-1} = \tr\; R_+ - \tr\; R_- = \int_\R d\omega_+ - \int_\R d\omega_- = \int_\R d\omega,
 \end{align*}
 in view of the previously mentioned lemma in \cite[Section~XI.4]{RS2}.
\end{proof}

More generally, for each $z\in\rho(T)$ the resolvent is given by
\begin{align}\label{eqnResolvent}
 (T-z)^{-1} g(x) = \int_\R G(z,x,s) g(s) d\om(s), \quad x\in\R,~ g\in\Lrom,
\end{align}
where $G$ is the Green function
\begin{align*}
 G(z,x,y) = W(z)^{-1} \begin{cases}
             \phi_-(z,x) \phi_+(z,y), & y\geq x, \\
             \phi_-(z,y) \phi_+(z,x), & y\leq x.
            \end{cases}
\end{align*}
 In fact, this can be shown following literally the proof of~\cite[Theorem~8.3]{measureSL} since the solution $\phi_\pm(z,\cdot\,)$ lies in the domain of $T$ near $\pm\infty$.
   The measure $\omega$ can be read off from the expansion of the Green function near zero on the diagonal.

\begin{lemma}\label{lemGdotzero}
For every $x\in\R$ we have 
 \begin{align}
  G(z,x,x) = 1+  z \int_\R \E^{-|x-s|} d\om(s) + \OO\left(z^2\right),
 \end{align}
 as $|z|\rightarrow 0$ in $\C$.
\end{lemma}

\begin{proof}
For each $x\in\R$ we get from the Neumann series~\eqref{eqnNeumannseries}
\begin{align}\label{eqnExpm}
 m_\pm(z,x) & = 1 \pm z \int_{x}^{\pm\infty} \left(\E^{\pm(x-s)}-1\right) d\om(s) + \OO\left(z^2\right),
\end{align}
 as $z\rightarrow 0$ in $\C$ and hence  
\begin{align*}
 \phi_-(z,x)\phi_+(z,x) & = 1 + z \int_\R \E^{-|x-s|}d\omega(s) - z\int_\R d\omega + \OO\left(z^2\right),
\end{align*}
 as $z\rightarrow 0$ in $\C$. Since this first expansion holds uniformly for all $x\in\R$, we get from equation~\eqref{eqnIntEqnDer} 
\begin{align*}
 m_\pm'(z,x) & = z \int_x^{\pm\infty} \E^{\pm(x-s)} d\om(s) + \OO\left(z^2\right),
\end{align*}
 for every $x\in\R$ as $z\rightarrow 0$ in $\C$.
 Therefore we have 
\begin{align*}
 W(z) & = m_+(z,x) m_-'(z,x) - m_+'(z,x)m_-(z,x) + m_-(z,x) m_+(z,x) \\
            & = 1 - z \int_\R d\om + \OO\left(z^2\right),
\end{align*}
proving the first claim. In particular we have
\begin{align*}
 W(z)^{-1} = 1 + z \int_\R d\om + \OO\left(z^2\right),
\end{align*}
and thus we finally get
\begin{align*}
 \frac{\phi_-(z,x)\phi_+(z,x)}{W(z)} & = 1 + z \int_\R \E^{-|x-s|}d\omega(s) + \OO\left(z^2\right),
\end{align*}
for every $x\in\R$ as $z\rightarrow 0$ in $\C$.
\end{proof}

Note that the quantity 
\begin{align}\label{defy}
 u(x) = \frac{1}{2} \int_\R \E^{-|x-s|} d\om(s), \quad x\in\R
\end{align}
appearing in Lemma~\ref{lemGdotzero} is important in view of applications to the Camassa--Holm equation, since it is the unique solution of $u-u_{xx}=\omega$ in $H^1(\R)$.

\section{Exponential growth of solutions}\label{secETZ}

In order to apply the inverse uniqueness result from~\cite{LeftDefiniteSL} we need to show that the solution $\phi_\pm$ is actually of exponential type zero, that is, 
\begin{align*}
 \ln^+|\phi_\pm(z,c)| = \oo(|z|), 
\end{align*}
as $|z|\rightarrow\infty$ in $\C$ for every $c\in\R$. To this end we fix some $c\in\R$ and restrict $T$ to the intervals $I_{c,-}=(-\infty,c)$ and $I_{c,+}=[c,\infty)$ by imposing a Dirichlet boundary condition at $c$. More precisely, the differential operator $T_{c,\pm}$ in the Hilbert space $L^2(I_{c,\pm};|\omega|)$ is given by 
\begin{align*}
 \dom{T_{c,\pm}} = \lbrace f\in\Deftau \,|\, f,\,\tau f\in L^2(I_{c,\pm};|\omega|),~\BC_{\pm\infty}(f) = f(c)=0 \rbrace
\end{align*}
and $T_{c,\pm} f = \tau f$ for $f\in\dom{T_{c,\pm}}$. Hereby note that the operator $T_{c,+}$ is actually multi-valued provided that $\omega(\lbrace c\rbrace)\not=0$ (but see \cite[Corollary~7.4]{measureSL} for details). Nevertheless, if we denote with $J_{c,\pm}$ the restriction of $J$ to $L^2(I_{c,\pm};|\omega|)$, then the (possibly multi-valued) operator $J_{c,\pm} T_{c,\pm}$ turns out to be self-adjoint in $L^2(I_{c,\pm};|\omega|)$. 
 Moreover, we will write $T_c=T_{c,-}\oplus T_{c,+}$ for the corresponding (in general multi-valued) operator in $L^2(I_{c,-};|\omega|)\oplus L^2(I_{c,+};|\omega|) = \Lrom$.

As before, the existence of the real entire solution $\phi_\pm$ guarantees that the spectrum of this operator $T_{c,\pm}$ is purely discrete and that its eigenvalues are precisely the zeros of the entire function $\phi_\pm(\,\cdot\,,c)$
\begin{align}\label{eqnspecTDc}
 \sigma(T_{c,\pm}) = \{ \mu\in\C \,|\, \phi_\pm(\mu,c)=0\}.
\end{align}
Moreover, this spectrum is real in view of equation~\eqref{eqnintpar}.

\begin{theorem}\label{thmPHIrep}
 The solution $\phi_\pm$ is of exponential type zero and given by 
\begin{align}
 \phi_\pm(z,x) = \E^{\mp\frac{x}{2}} \prod_{\mu\in \sigma(T_{x,\pm})} \biggl(1-\frac{z}{\mu}\biggr), \quad z\in\C,~x\in\R.
\end{align}
\end{theorem}

\begin{proof}
First of all, \cite[Theorem~8.3]{measureSL} shows that the inverse of the (possibly multi-valued) operator $J_{c,\pm} T_{c,\pm}$ is given by 
\begin{align*}
 (J_{c,\pm} T_{c,\pm})^{-1} g(x) = \pm \int_{c}^{\pm\infty} \left( \E^{-\frac{|x-s|}{2}} - \E^{\pm \left(c - \frac{x+s}{2}\right)} \right) g(s)d|\omega|(s), \quad x\in I_{c,\pm}
\end{align*}
for each $g\in L^2(I_{c,\pm};|\omega|)$. 
Now as in the proof of Proposition~\ref{propInverse} one shows that $(J_{c,\pm} T_{c,\pm})^{-1}$ and hence also $T_{c,\pm}$ is a trace class operator with trace given by
\begin{align*}
 \sum_{\mu\in \sig(T_{c,\pm})} \frac{1}{\mu} = \tr\; (T_{c,\pm})^{-1} = \pm \int_{c}^{\pm\infty} \left(1-\E^{\pm(c-s)}\right) d\omega(s).
\end{align*}
Moreover, since $\phi_\pm(\cdot,c)$ is of finite exponential type with summable zeros (note that all of them are simple in view of~\eqref{eqnExpm}), the Hadamard factorization shows that 
\begin{align}\tag{$*$}\label{eqnPHIreppre}
 \phi_\pm(z,c) = \phi_\pm(0,c) \E^{A_{c,\pm} z} \prod_{\mu\in \sigma(T_{c,\pm})} \biggl( 1-\frac{z}{\mu}\biggr), \quad z\in\C,
\end{align}
for some $A_{c,\pm}\in\R$.
 Using the Neumann series~\eqref{eqnNeumannseries} near zero and the representation~\eqref{eqnPHIreppre} on the other side we get
\begin{align*}
\pm \int_{c}^{\pm\infty} \left(1-\E^{\pm(c-s)}\right)d\omega(s) = -\frac{\dot{\phi}_\pm(0,c)}{\phi_\pm(0,c)} = \sum_{\mu\in \sigma(T_{c,\pm})}\frac{1}{\mu} - A_{c,\pm}.
\end{align*}
Since the integral on the left-hand side is equal to the trace of $T_{c,\pm}^{-1}$ we conclude that $A_{c,\pm}=0$, which yields the claimed representation for $\phi_\pm(\,\cdot\,,c)$. In particular, this shows that $\phi_\pm(\,\cdot\,,c)$ is of exponential type zero. 
\end{proof}

An analogous definition as above can be made with a Neumann boundary condition at $c$ and the corresponding operator $T'_{c,\pm}$ is given by
\begin{align*}
 \dom{T'_{c,\pm}} = \lbrace f\in\Deftau \,|\, f,\, \tau f\in L^2(I_{c,\pm};|\omega|),~\BC_{\pm\infty}(f) = f'(c)=0 \rbrace
\end{align*}
and $T'_{c,\pm} f = \tau f$ for $f\in\dom{T'_{c,\pm}}$. Its eigenvalues are precisely the zeros of the real entire functions $\phi_\pm'(\,\cdot\,,c)$
\begin{align}\label{eqnspecTNc}
 \sigma(T'_{c,\pm}) = \{ \nu\in\C \,|\, \phi_\pm'(\nu,c)=0\},
\end{align}
which are real for a similar reason as above.
 Now in much the same manner as in the proof of Theorem~\ref{thmPHIrep} one may show that the entire function $\phi_\pm'(\,\cdot\,,c)$ is of exponential type zero as well with
\begin{align}
 \phi_\pm'(z,x) = \mp \frac{1}{2} \E^{\mp\frac{x}{2}} \prod_{\nu\in \sigma(T'_{x,\pm})} \biggl( 1-\frac{z}{\nu}\biggr), \quad z\in\C,~x\in\R.
\end{align} 
 From these results we also get a product representation for the Wronskian $W$.

\begin{corollary}\label{corWrep}
 The Wronskian $W$ has the product representation
\begin{align}
 W(z) = \prod_{\lambda\in\sigma(T)} \biggl(1-\frac{z}{\lambda}\biggr), \quad z\in\C.
\end{align}
\end{corollary}

\begin{proof}
 Since $W$ is of exponential type zero with $W(0) = 1$ and summable roots, this follows from the Hadamard factorization.
\end{proof}

From the product representations of the functions in Theorem~\ref{thmPHIrep} and Corollary~\ref{corWrep} one sees that it is possible to express the quantity  in~\eqref{defy} in terms of the spectra $\sigma(T_{x,-})$, $\sigma(T_{x,+})$, $x\in\R$ and $\sigma(T)$.

\begin{corollary}\label{coruthreespec} 
 For each $x\in\R$ we have
\begin{align}
  \int_\R \E^{-|x-s|}d\omega(s) = \sum_{\lambda\in\sigma(T)} \frac{1}{\lambda} - \sum_{\mu\in \sigma(T_{x,-})} \frac{1}{\mu} - \sum_{\mu\in \sigma(T_{x,+})} \frac{1}{\mu}, \quad x\in\R.
\end{align}
\end{corollary}

\begin{proof}
 This follows from Lemma~\ref{lemGdotzero} and the product representations of the functions in Theorem~\ref{thmPHIrep} and Corollary~\ref{corWrep}.
\end{proof}

\section{The left-definite operator}\label{secLD}

In this section we will introduce the left-definite operator, associated with the isospectral problem~\eqref{eqnISOPdiffeqn}. 
Therefore consider the Sobolev space $H^1(\R)$, equipped with the modified inner product
\begin{align*}
 \spr{f}{g}_{H^1(\R)} = \frac{1}{4} \int_\R f(x) g(x)^\ast dx + \int_\R f'(x) g'(x)^\ast dx, \quad f,\, g\in H^1(\R).
\end{align*}
We define the (in general multi-valued) operator $S$ in $H^1(\R)$ by specifying its graph to be 
\begin{align*}
 \lbrace (f,f_\tau)\in H^1(\R) \times H^1(\R) \,|\, f\in\Deftau,~\tau f = f_\tau ~\text{in }L_{\mathrm{loc}}^1(\R;|\omega|)\rbrace.
\end{align*}
From results in~\cite[Section~3]{LeftDefiniteSL} it follows that $S$ is a self-adjoint operator which is multi-valued
unless the support of $\omega$ is dense. In this case we can still obtain a single-valued operator, the operator
part of $S$, by restricting $S$ to the closure of its domain.
Next, since $S$ is a self-adjoint realization of the differential equation~\eqref{eqnISOPdiffeqn}, it is not surprising that the spectral properties of $S$ and $T$ are very similar. 

 \begin{proposition}\label{propInverseS}
  The linear operator $S$ has the same spectrum as $T$ and its inverse is given by
  \begin{align}\label{eqnSinv}
   S^{-1} g(x) = \int_\R \E^{-\frac{|x-s|}{2}} g(s) d\omega(s), \quad x\in\R,~ g\in H^1(\R).
  \end{align}
 \end{proposition}

 \begin{proof}
  From Theorem~\ref{thmSpecResPHI} and the remark after~\cite[Lemma~4.5]{LeftDefiniteSL} we infer that $S$ has purely discrete spectrum. 
 In fact, since $\phi_\pm(z,\cdot\,)$, $z\in\C$ are real entire solutions which lie in $H^1(\R)$ near $\pm\infty$, each associated singular Weyl--Titchmarsh function (see \cite[Section~4]{LeftDefiniteSL}) is meromorphic in $\C$ with poles contained in $\sigma(T)\cup\lbrace 0\rbrace$. 
 Now in view of \cite[Equation~(4.5)]{LeftDefiniteSL}, for each $g$ in a dense subspace of $H^1(\R)$, the function $z\mapsto\spr{(S-z)^{-1}g}{g}$ is meromorphic in $\C$ as well with poles contained in $\sigma(T)\cup\lbrace 0\rbrace$, which shows that $S$ has purely discrete spectrum. 
 Moreover, the fact that zero is not an eigenvalue of $S$ follows from~\cite[Proposition~2.7]{LeftDefiniteSL}. Now since for each $\lambda\in\R$ the solution $\phi_\pm(\lambda,\cdot\,)$ lies in the domain of $S$ near $\pm\infty$, we infer that $\lambda$ is an eigenvalue of $S$ if and only if these solutions are linearly dependent. Hence the spectra of $S$ and $T$ are equal. Finally, if $g\in H^1(\R)$ has compact support, then $S^{-1} g$ is given as in the claim since the function on the right-hand side of~\eqref{eqnSinv} is a solution of $\tau f = g$ which lies in $H^1(\R)$. The general case, when $g\in H^1(\R)$ follows from continuity of both sides in~\eqref{eqnSinv}.
 \end{proof}

  Although, this is all we need in order to apply the inverse uniqueness result from~\cite{LeftDefiniteSL} to the operator $T$, we will furthermore show how $S$ and $T$ are related. Therefore recall that with each strictly positive self-adjoint operator in a Krein space one can associate a so-called left-definite operator in some Hilbert space; see \cite[Section~11.4]{zettl} for a discussion which is close to our situation. The left-definite Hilbert space $H_1$ associated with $T$ is the domain of $\sqrt{JT}$ equipped with the inner product
 \begin{align*}
  \spr{f}{g}_1 = \spr{\sqrt{JT}f}{\sqrt{JT}g}_{\Lrom}, \quad f,\, g\in H_1.
 \end{align*}
 Furthermore, the left-definite operator $S_1$ is obtained by restricting $T$ to the space $H_1$. More precisely $S_1$ is given by
 \begin{align*}
  \dom{S_1} = \ran{T^{-1}|_{H_1}},
 \end{align*}
 and $S_1 f = T f$ for $f\in\dom{S_1}$. It turns out that this operator $S_1$ is self-adjoint in $H_1$. In particular, note that its domain and hence also the domain of $T$ are dense in $H_1$. Moreover, it is known that the spectra of $T$ and $S_1$ are the same. We will now show that one may identify $H_1$ with a closed subspace of $H^1(\R)$ and that the operator $S_1$ is essentially the same as the linear operator $S$ defined above.
 
 \begin{proposition}
  The operator part of $S$ is unitarily equivalent to $S_1$.
 \end{proposition}
 
\begin{proof}
  First of all note that $\dom{T}$ may be regarded as a subset of $H^1(\R)$. In fact, each $f\in\dom{T}$ can be written as
  \begin{align*}
   f(x) = \int_\R \E^{-\frac{|x-s|}{2}} g(s)d\omega(s), \quad x\in\R,
  \end{align*} 
  for some $g\in\Lrom$. It is not hard to show that this function actually lies in $H^1(\R)$.
  Moreover, an integration by parts shows that
 \begin{equation}
  \begin{split}\label{eqnLDNormRel}
  \spr{f}{g}_1 & =\spr{f}{JT g}_{\Lrom} = \int_\R f(x) \tau g(x)^\ast d\omega(x) \\
               & = \frac{1}{4} \int_\R f(x)g(x)^\ast dx + \int_\R f'(x) g'(x)^\ast dx  = \spr{f}{g}_{H^1(\R)},
  \end{split}
 \end{equation}
 for each $f$, $g\in\dom{T}$. 
 Hence $\dom{T}$ is even isometrically embedded in $H^1(\R)$ and thus the Hilbert space $H_1$ can be identified with a closed subspace of $H^1(\R)$. 
 Now given some $f\in\dom{T}$ and $g\in\mul{S}$ we have
 \begin{align*}
  \spr{f}{g}_{H^1(\R)} = \spr{T f}{g}_{\Krom} = 0,
 \end{align*}
 since $g$ vanishes almost everywhere with respect to $|\omega|$ (see \cite[Proposition~2.5]{LeftDefiniteSL}). As a consequence $\dom{T}$ and hence also the space $H_1$ are contained in the closure of $\dom{S}$. On the other side, given some $f\in\dom{S}$ there is some function $g\in H^1(\R)$ such that
 \begin{align*}
  f(x) = \int_\R \E^{-\frac{|x-s|}{2}} g(s) d\omega(s), \quad x\in\R.
 \end{align*}
 But since $g$ also lies in $\Lrom$ we infer from Proposition~\ref{propInverse} that $f$ also lies in $\dom{T}$. Thus we see that $H_1$ actually is the closure of the domain of $S$. Finally, since the inverses of $S_1$ and the operator part of $S$ are given as in Proposition~\ref{propInverse} and Proposition~\ref{propInverseS} we infer that they are equal.
 \end{proof}

Finally, as a simple consequence of the results in~\cite{LeftDefiniteSL} we get the following inverse uniqueness theorem for the left/right spectral measure $\rho_\pm$ of $T$ given by
\begin{align*}
 \rho_\pm = \sum_{\lambda\in\sigma(T)} \gamma_{\lambda,\pm}^{-2} \delta_\lambda,
\end{align*}
where $\delta_\lambda$ is the Dirac measure in the point $\lambda$.

\begin{theorem}\label{thmInvUniq}
The measure $\omega$ is uniquely determined by the spectral measure $\rho_\pm$.
\end{theorem}

\begin{proof}
This follows by applying~\cite[Theorem~7.5]{LeftDefiniteSL} to the multi-valued operator $S$. The assumptions of this theorem are readily verified. Also note that the positive discrete measure
\begin{align*}
 \sum_{\lambda\in\sigma(T)} \lambda^{-1} \gamma_{\lambda,\pm}^{-2} \delta_\lambda,
\end{align*}
is the left/right spectral measure associated with $S$ in view of~\eqref{eqnNormConstH1}.
\end{proof}

Furthermore, the product representations in the previous section as well as the relations~\eqref{eqnCoup} and~\eqref{eqnWlam} immediately yield the following inverse uniqueness result from three spectra. 

\begin{corollary}
 For each fixed $c\in\R$, the measure $\omega$ is uniquely determined by the three spectra $\sigma(T)$, $\sigma(T_{c,-})$ and $\sigma(T_{c,+})$ provided they are disjoint.
\end{corollary}

\begin{proof}
In this case the left/right norming constant $\gamma_{\lambda,\pm}^2$ for some given eigenvalue $\lambda\in\sigma(T)$ can be written down explicitly in terms of these three spectra 
\begin{align}\label{eqngamma}
 \gamma_{\lambda,\pm}^2 
 &=  \frac{\E^{\mp c}}{\lambda} \prod_{\kappa\in\sigma(T)\backslash\lbrace\lambda\rbrace} \biggl(1-\frac{\lambda}{\kappa}\biggr)  \prod_{\mu\in\sigma(T_{c,\pm})}\biggl(1-\frac{\lambda}{\mu}\biggr) \prod_{\mu\in\sigma(T_{c,\mp})} \biggl(1-\frac{\lambda}{\mu}\biggr)^{-1}
\end{align}
  using the equations~\eqref{eqnCoup} and~\eqref{eqnWlam}, Theorem~\ref{thmPHIrep} and Corollary~\ref{corWrep}.
\end{proof}

Similarly, for every $c\in\R$ there is also a left-definite operator in the Sobolev space $H_0^1(I_{c,\pm})$ (with modified norm) corresponding to the operator $T_{c,\pm}$. This (in general multi-valued) self-adjoint operator $S_{c,\pm}$ is again defined via its graph 
\begin{align*}
  \lbrace (f,f_\tau)\in H_0^1(I_{c,\pm})\times H_0^1(I_{c,\pm}) \,|\, f\in\Deftau,~\tau f = f_\tau ~\text{in }L_{\mathrm{loc}}^1(I_{c,\pm};|\omega|)\rbrace.
\end{align*}
 Spectral theory for $S_{c,\pm}$ is closely related to spectral theory for the operator $T_{c,\pm}$.

\begin{proposition}
 For each $c\in\R$ the operator $S_{c,\pm}$ has the same spectrum as $T_{c,\pm}$ and its inverse is given by
 \begin{align}\label{eqnSDinv}
  S_{c,\pm}^{-1} g(x) = \pm \int_c^{\pm\infty} \left( \E^{-\frac{|x-s|}{2}} - \E^{\pm\left(c-\frac{x+s}{2}\right)}\right) g(s)d\omega(s), \quad x\in I_{c,\pm},
 \end{align}
 for all functions $g\in H_0^1(I_{c,\pm})$.
\end{proposition}

\begin{proof}
  From~\cite[Lemma~4.5]{LeftDefiniteSL} we infer that $S_{c,\pm}$ has purely discrete spectrum with zero being in the resolvent set in view of~\cite[Proposition~2.7]{LeftDefiniteSL}. 
  Now since for each $\lambda\in\R$ the solution $\phi_\pm(\lambda,\cdot\,)$ lies in the domain of $S_{c,\pm}$ near $\pm\infty$, we infer that $\lambda$ is an eigenvalue of $S_{c,\pm}$ if and only if $\phi_\pm(\lambda,c)=0$. Hence the spectra of $S_{c,\pm}$ and $T_{c,\pm}$ are the same. If $g\in H_0^1(I_{c,\pm})$ has compact support, then $S_{c,\pm}^{-1} g$ is given as in the claim since the function on the right-hand side of~\eqref{eqnSDinv} is a solution of $\tau f = g$ which lies in $H_0^1(I_{c,\pm})$. The general case, when $g\in H_0^1(I_{c,\pm})$ follows from continuity of both sides in~\eqref{eqnSDinv}.
\end{proof}

Finally note that the Fredholm determinants of the inverses of these left-definite operators are the entire functions
\begin{align}\label{eqnFD}
 W(z) = \det(I-zS^{-1}) \quad\text{and}\quad \phi_\pm(c,z) = \E^{\mp\frac{c}{2}} \det(I-zS^{-1}_{c,\pm}), \quad z\in\C,
\end{align}
 as the product representations of these functions show.

\section{Continuity and compactness}\label{secCC}

In this section let $\M$ be the set of all finite signed Borel measures on $\R$, equipped with the weak$^\ast$ topology, that is, the initial topology with respect to the functionals 
 \begin{align*}
  \omega \mapsto \int_\R f d\omega, \quad f\in C_0(\R)
 \end{align*}
 on $\M$, where $C_0(\R)$ is the space of continuous functions which vanish at infinity. 
  Each sequence $\omega_n\in\M$, $n\in\N$ which converges to $\omega$ with respect to this topology is known to be 
 uniformly bounded in total variation and we write $\omega_n\rightharpoonup^\ast\omega$.

\begin{lemma}\label{lemCont}
 If $\omega_n\rightharpoonup^{\ast}\omega$, then the corresponding operators $S_n^{-1}$ converge strongly to $S^{-1}$ and the operators $S_{n,c,\pm}^{-1}$ converge strongly to $S_{c,\pm}^{-1}$ for each $c\in\R$. 
\end{lemma}

\begin{proof}
 Given some arbitrary $g\in H^1(\R)$, we have for each $x\in\R$ 
 \begin{align*}
  S_{n}^{-1}g(x) = \int_\R \E^{-\frac{|x-s|}{2}} g(s)d\omega_n(s) \rightarrow \int_\R \E^{-\frac{|x-s|}{2}} g(s)d\omega(s) = S^{-1}g(x),
 \end{align*}
 as $n\rightarrow\infty$. Since point evaluations are dense and these operators are uniformly bounded, this implies that $S_{n}^{-1}$ converges to $S^{-1}$ in the weak operator topology. 
 In order to prove that they also converge in the strong operator topology, note that an integration by parts shows
 \begin{align*}
  \|S_{n}^{-1}g\|_{H^1(\R)}^2 = \spr{ S_{n}^{-1}g}{g}_{L^2(\R;\omega_n)}, \quad n\in\N.
 \end{align*}
 Moreover, since the functions $S_{n}^{-1}g$ are uniformly bounded in $H^1(\R)$, we infer from the Arzel\`{a}--Ascoli theorem that $S_n^{-1} g$ converges to $S^{-1}g$ locally uniformly.
 Now if we assume that $g$ has compact support, then we have
 \begin{align*}
  & \left| \int_\R S_{n}^{-1}g(x) g(x)^\ast d\omega_n(x) - \int_\R S^{-1}g(x) g(x)^\ast d\omega(x) \right| \\
  & \qquad\qquad\qquad \leq  \|g\|_{L^\infty(\R)} \int_\R d|\omega_n| \sup_{x\in\supp(g)} |S_n^{-1}g(x)-S^{-1}g(x)| \\
  & \qquad\qquad\qquad\quad + \left|\int_\R S^{-1}g(x)g(x)^\ast d\omega_n(x) - \int_\R S^{-1}g(x)g(x)^\ast d\omega(x)\right|,
 \end{align*}
 and thus $S_n^{-1} g$ converges to $S^{-1}g$ in $H^1(\R)$. Finally, since the inverses are uniformly bounded we infer that $S_n^{-1}$ converges to $S^{-1}$ in the strong operator topology.
 The convergence of the operators $S_{n,c,\pm}^{-1}$ may be verified similarly.  
\end{proof}

  In particular, strong convergence implies that each eigenvalue $\lambda\in\sigma(S)$ is the limit of some sequence $\lambda_n\in\sigma(S_n)$, $n\in\N$. 
 However, convergence of the weight measures in the weak$^\ast$ topology does in general not imply convergence of the corresponding spectral measures (in any reasonable topology).  
 Nevertheless, restricted to certain subsets of $\M$, it will be possible to describe the behavior of the spectral measures to some extent. 
 Therefore fix some discrete set $\sigma\subseteq\R$ such that  
 \begin{align}\label{eqnEVsumabl}
   \sum_{\lambda\in\sigma} \frac{1}{|\lambda|} < \infty
 \end{align}
 and consider the set $\M_\sigma$ of all finite signed Borel measures on $\R$ whose associated spectra are contained in $\sigma$. 
  From the strong convergence in Lemma~\ref{lemCont} we infer that $\M_\sigma$ is closed with respect to the weak$^\ast$ topology. 
  As a first step we show how the functions in~\eqref{eqnFD} behave under weak$^\ast$ convergence of weight measures in $\M_\sigma$.

 \begin{lemma}\label{lemFCont}
  Suppose that the measures $\omega_n$, $n\in\N$ lie in $\M_\sigma$ and $\omega_n\rightharpoonup^{\ast}\omega$. Then for each $c\in\R$ there is a subsequence $\omega_{n_k}$ and disjoint sets $\sigma_-$, $\sigma_+\subseteq\sigma\backslash\sigma(S)$ such that the functions $W_{n_k}$  
 and $\phi_{n_k,\pm}(\,\cdot\,,c)$ converge locally uniformly to  
 \begin{align}\label{eqnWPHIinf}
  W(z) \prod_{\lambda\in\sigma_-\cup\sigma_+} \biggl(1-\frac{z}{\lambda}\biggr) \quad\text{and}\quad \phi_\pm(z,c) \prod_{\lambda\in\sigma_\pm} \biggl(1-\frac{z}{\lambda}\biggr), \quad z\in\C,
 \end{align}
 respectively as $k\rightarrow\infty$.
 \end{lemma}

 \begin{proof}
  First of all note that the functions $W_n$ and $\phi_{n,\pm}(\,\cdot\,,c)$, $n\in\N$ are uniformly bounded by a scalar multiple of the product 
   \begin{align*}
      \prod_{\lambda\in\sigma} \biggl(1+\frac{|z|}{|\lambda|}\biggr), \quad z\in\C.
   \end{align*} 
   In fact, this follows from the interlacing property of the (necessarily simple) zeros and poles of the meromorphic Herglotz--Nevanlinna function \cite[Proposition~4.4]{LeftDefiniteSL}
   \begin{align*}
    zG(z,c,c) = \left(\frac{\phi_-'(z,c)}{z\phi_-(z,c)} - \frac{\phi_+'(z,c)}{z\phi_+(z,c)}\right)^{-1}, \quad z\in\C\backslash\R.
   \end{align*}
  Hence there is a subsequence $\omega_{n_k}$ such that the functions $W_{n_k}$ and $\phi_{n_k,\pm}(\,\cdot\,,c)$ converge locally uniformly to some entire functions of exponential type zero. Moreover, since the zeros of the functions $W_{n_k}$ are contained in $\sigma$, their limit is of the form  
  \begin{align*}
   \lim_{k\rightarrow\infty} W_{n_k}(z) = \prod_{\lambda\in\sigma_\infty} \biggl(1-\frac{z}{\lambda}\biggr), \quad z\in\C
  \end{align*}
  for some set $\sigma_\infty\subseteq\sigma$ which contains the spectrum of $S$ in view of Lemma~\ref{lemCont}. 
   Similarly, the functions $\phi_{n_k,\pm}(\,\cdot\,,c)$ converge to some canonical products which vanish in the points of $\sigma(S_{c,\pm})$ respectively, i.e., to functions as given in~\eqref{eqnWPHIinf}.
  But the sets $\sigma_\pm$ of additional zeros of these limits form a partition of $\sigma_\infty\backslash\sigma(S)$. 
 In fact, the convergence in Lemma~\ref{lemCont} implies that the corresponding Green functions $G_n(\,\cdot\,,c,c)$ from \eqref{eqnResolvent} converge to $G(\,\cdot\,,c,c)$ locally uniformly in $\C\backslash\R$ (see \cite[Theorem~3.5]{LeftDefiniteSL}) which proves the claim.  
 \end{proof}

 As a simple consequence we are able to describe the behavior of the norming constants and hence of the spectral measures under weak$^\ast$ convergence in $\M_\sigma$.

 \begin{theorem}\label{thmweakstarconv}
  Suppose that the measures $\omega_n$, $n\in\N$ lie in $\M_\sigma$ and $\omega_n\rightharpoonup^\ast\omega$. Then there is a subsequence $\omega_{n_k}$ such that $\sigma(S_{n_k})$ converges to some $\sigma_\infty\subseteq\sigma$ as $k\rightarrow\infty$ and  
  \begin{align}
   \lim_{k\rightarrow\infty} \lambda\gamma_{n_k,\lambda,\pm}^2 = \begin{cases} 0, &  \lambda\in\sigma_\pm, \\ \lambda \gamma_{\lambda,\pm}^2 \prod_{\kappa\in\sigma_\pm} \left(1-\frac{\lambda}{\kappa}\right)^2, & \lambda\in\sigma(S), \\ \infty, & \lambda\in\sigma_\mp,  \end{cases}
  \end{align}
  for some partition $\sigma_-\dot{\cup}\,\sigma_+$ of $\sigma_\infty\backslash\sigma(S)$. 
 \end{theorem}
 
  \begin{proof} 
    First of all note that one may choose some $c\in\R$ such that for each $n\in\N$ the three spectra $\sigma(S_n)$, $\sigma(S_{n,c,-})$ and $\sigma(S_{n,c,+})$ are disjoint. 
     In fact, this fails only if $\phi_{n,\pm}(\lambda,c)=0$ for some $\lambda\in\sigma(S_n)$, $n\in\N$ and since each of the (countably many) functions $\phi_{n,\pm}(\lambda,\cdot\,)$, $\lambda\in\sigma(S_n)$, $n\in\N$ has only countably many zeros, it is possible to choose a $c\in\R$ with the claimed property. 
     For the same reason we may as well assume that $\sigma(S)$, $\sigma(S_{c,-})$ and $\sigma(S_{c,+})$ are disjoint and hence the claim immediately follows from Lemma~\ref{lemFCont} and~\eqref{eqngamma}.
 \end{proof}

 Finally, using compactness arguments it is also possible to deduce some kind of continuity for the inverse spectral problem. 
 In fact,  pointwise convergence of the norming constants implies weak$^\ast$ convergence of the corresponding weight measures in $\M_\sigma$, provided they are uniformly bounded.
 
 \begin{corollary}\label{corInvCont}
  Suppose that the measures $\omega_n$, $n\in\N$ lie in $\M_\sigma$, that $\sigma(S_n)$ converges to some $\sigma_\infty\subseteq\sigma$ and that for each $\lambda\in\sigma_\infty$ the quantities  $\lambda \gamma_{n,\lambda,\pm}^2$ converge to some $\lambda \gamma_{\infty,\lambda,\pm}^2\in[0,\infty]$.
  If the measures $\omega_n$ are uniformly bounded, then $\omega_n\rightharpoonup^\ast\omega$ for some $\omega\in\M_\sigma$ with spectrum $\sigma_\infty\backslash(\sigma_-\cup\,\sigma_+)$ and norming constants  
  \begin{align*}
   \gamma_{\infty,\lambda,\pm}^2 \prod_{\kappa\in\sigma_\pm} \biggl(1-\frac{\lambda}{\kappa}\biggr)^{-2}, \quad \lambda\in\sigma_\infty\backslash(\sigma_-\cup\,\sigma_+),
  \end{align*}
  where $\sigma_\pm=\lbrace \lambda\in\sigma_\infty \,|\, \lambda \gamma_{\infty,\lambda,\pm}^2 = 0\rbrace$ and $\sigma_\mp=\lbrace \lambda\in\sigma_\infty \,|\, \lambda \gamma_{\infty,\lambda,\pm}^2 = \infty\rbrace$. 
 \end{corollary}

 \begin{proof}
  By compactness of $\M_\sigma$, each subsequence of $\omega_n$ has a weak$^\ast$ convergent subsequence. From Theorem~\ref{thmweakstarconv} we infer that the spectrum and the norming constants corresponding to the limit of this subsequence are given as in the claim. 
 Now the inverse uniqueness result Theorem~\ref{thmInvUniq} shows that all such subsequences actually converge to the same $\omega\in\M_\sigma$ and hence $\omega_n\rightharpoonup^\ast\omega$.
 \end{proof}

 Note that $\M_\sigma$ is bounded (and hence compact) if and only if $\sigma$ is positive or negative (by this we mean contained in $\R^+$ or $\R^-$). 
  In fact, if $\sigma$ is positive or negative, then \eqref{eqnTF} shows that $\M_\sigma$ is bounded. 
  For the converse, first observe that the Wronskian corresponding to some weight measure of the form 
   $\omega = \omega_+ \delta_{\varepsilon} - \omega_- \delta_{-\varepsilon}$ with $\omega_+$, $\omega_-\in\R$ and $\varepsilon>0$ is given by
  \begin{align}\label{eqnTwoPeakonWronski}
   W(z) = 1 - z(\omega_+ - \omega_-) - z^2 \omega_+ \omega_- (1-\E^{-2\varepsilon}), \quad z\in\C.
  \end{align}
  Now suppose there is some positive $\lambda_+\in\sigma$ and some negative $\lambda_-\in\sigma$. 
  Then \eqref{eqnTwoPeakonWronski} shows that for each large enough given $\omega_+$ one may choose $\omega_-\in\R$, $\varepsilon>0$ such that the eigenvalues corresponding to the weight measure $\omega$ are $\lambda_+$ and $\lambda_-$. 
 But this guarantees that in this case $\M_\sigma$ is unbounded indeed.

\section{Spectral transformation}\label{secIP}

Since $S$ is a self-adjoint operator in $H^1(\R)$, it is clear that the eigenfunctions $\phi_\pm(\lambda,\cdot\,)$, $\lambda\in\sigma(S)$ are orthogonal in $H^1(\R)$ and that their span is dense in $\dom{S}$. 
 Moreover, because $T$ is self-adjoint in the Krein space $\Krom$, they are also orthogonal with respect to the indefinite inner product there. 
 The considerations in the previous section also show that their span is even dense in $\Krom$. 
 More precisely, from~\eqref{eqnLDNormRel} and strict positivity of $JT$ we get 
\begin{align*}
 \eps \|f\|_{\Lrom}^2 \leq \spr{JTf}{f}_{\Lrom} = \|f\|_{H^1(\R)}^2, \quad f\in\dom{T}
\end{align*}
for some $\eps>0$, which shows that the span of all functions $\phi_\pm(\lambda,\cdot\,)$, $\lambda\in\sigma(T)$ is dense in $\dom{T}$ and hence also in $\Krom$.
  
  Next, for each function $f\in \Krom$ we define the transform $\mathcal{F}_\pm f$ on $\sigma(T)$ by 
 \begin{align}
  \mathcal{F}_\pm f(\lambda) = \int_\R \phi_\pm(\lambda,x) f(x)d\omega(x), \quad \lambda\in\sigma(T). 
 \end{align}
 In order to state the following result, recall that some collection of functions in $\Lrom$ forms a Riesz basis if they are an orthonormal basis with respect to some inner product which is equivalent to the usual inner product in $\Lrom$.
 
  \begin{proposition}\label{propFtrans}
  There exists a Riesz basis in $\Lrom$ consisting of eigenfunctions if and only if $\mathcal{F}_\pm$ maps $\Krom$ bicontinuously onto $L^2(\R;\rho_\pm)$ with
  \begin{align}\label{eqnKiso}
   \spr{\mathcal{F}_\pm f}{\mathcal{F}_\pm g}_{L^2(\R;\rho_\pm)} = \spr{f}{g}_{\Krom}, \quad f,\, g\in \Krom.
  \end{align}
  In this case,  $\mathcal{F}_\pm$ maps $T$ onto multiplication with the independent variable.
 \end{proposition}
 
 \begin{proof}
  If there exists a Riesz basis consisting of eigenfunctions, then the direct sum  
   \begin{align}\label{eqnDirectSum}
    \overline{\linspan\lbrace\phi_\pm(\lambda,\cdot\,) \,|\, \lambda\in\sigma(T)\cap\R^+ \rbrace} \,\dot{+}\, \overline{\linspan\lbrace\phi_\pm(\lambda,\cdot\,) \,|\, \lambda\in\sigma(T)\cap\R^- \rbrace}
  \end{align}
  is orthogonal with respect to some (equivalent) inner product in $\Lrom$ and hence closed. 
  In particular, this sum is a fundamental decomposition of $\Krom$ and the corresponding projections are denoted by $P_{\pm}^+$ and $P_{\pm}^-$. 
  Then we have
  \begin{align*}
   \mathcal{F}_\pm f(\lambda) = \begin{cases} \mathcal{F}_\pm P_{\pm}^+ f(\lambda), & \lambda\in\sigma(T)\cap\R^+, \\                                                 \mathcal{F}_\pm P_{\pm}^- f(\lambda), &  \lambda\in\sigma(T)\cap\R^-, \end{cases}
   \end{align*}
 since the sum~\eqref{eqnDirectSum} is orthogonal with respect to the indefinite inner product in $\Krom$ as well. 
  Moreover, the functions $\phi_\pm(\lambda,\cdot\,)$, $\lambda\in\sigma(T)\cap\R^+$ are a complete orthogonal set of the positive definite subspace $\ran{P_{\pm}^+}\subseteq\Krom$ and $\phi_\pm(\lambda,\cdot\,)$, $\lambda\in\sigma(T)\cap\R^-$ are a complete orthogonal set of the negative definite subspace $\ran{P_{\pm}^-}\subseteq\Krom$. 
 Thus it is clear that $\mathcal{F}_\pm$ maps $\ran{P_{\pm}^+}$ unitarily onto $L^2(\R^+;\rho_\pm)$ and $\ran{P_{\pm}^-}$ unitarily onto $L^2(\R^-;\rho_\pm)$. 
  As a consequence, we infer that $\mathcal{F}_\pm$ is a bijection from $\Krom$ onto $L^2(\R;\rho_\pm)$ satisfying~\eqref{eqnKiso}. 
  Furthermore, if $f_n\rightarrow f$ in $\Krom$ and $\mathcal{F}_\pm f_n \rightarrow F$ in $L^2(\R;\rho_\pm)$ as $n\rightarrow\infty$, then 
 \begin{align*}
  \mathcal{F}_\pm f(\lambda) & = \spr{f}{\phi_\pm(\lambda,\cdot\,)}_{\Krom} = \lim_{n\rightarrow\infty} \spr{f_n}{\phi_\pm(\lambda,\cdot\,)}_{\Krom} = \lim_{n\rightarrow\infty} \mathcal{F}_\pm f_n(\lambda) 
    = F(\lambda)
 \end{align*}
 for each $\lambda\in\sigma(T)$, 
 which proves that $\mathcal{F}_\pm$ is bicontinuous. 
 For the converse, note that the eigenfunctions $\phi_\pm(\lambda,\cdot\,)$, $\lambda\in\sigma(T)$ are orthogonal with respect to the (positive definite) inner product given by 
 \begin{align*}
  (f,g) \mapsto \spr{\mathcal{F}_\pm f}{\mathcal{F}_\pm g}_{L^2(\R;|\rho_\pm|)}, \quad (f,g)\in\Lrom\times\Lrom,
 \end{align*}
 which is equivalent to the usual inner product in $\Lrom$ by assumption. 

   In order to prove the last claim we denote with $\mathrm{M}_{\mathrm{id}}$ the maximally defined operator of multiplication with the independent variable in $\Krom$.
    Then for each $\kappa\in\sigma(T)$ we have
  \begin{align*}
   \mathcal{F}_\pm T^{-1} \phi_\pm(\kappa,\cdot\,)(\lambda) = \kappa^{-1} \mathcal{F}_\pm \phi_\pm(\kappa,\cdot\,)(\lambda) = \mathrm{M}_{\mathrm{id}}^{-1} \mathcal{F}_\pm \phi_\pm(\kappa,\cdot\,)(\lambda), \quad \lambda\in\sigma(T)
  \end{align*}
  and hence $T^{-1}=\mathcal{F}_\pm^{-1} \mathrm{M}_{\mathrm{id}}^{-1}\mathcal{F}_\pm $ on a dense subspace. Now the claim follows since both sides are continuous linear operators on $\Krom$.
 \end{proof}
 
  It is a quite delicate question whether or not there is a Riesz basis consisting of eigenfunctions and we only refer to \cite{bifl} and the references cited there. 
  Also note that if there are only finitely many positive or negative eigenvalues, then there is always such a Riesz basis since then the direct sum in~\eqref{eqnDirectSum} is closed for sure.
  Next we will compute the transforms of some particular functions.
 
 \begin{lemma}\label{lemTransGreen}
  For each $z\in\rho(T)$ and $x\in\R$ the transform of the Green function $G(z,x,\cdot\,)$ is given by
  \begin{align}
   \mathcal{F}_\pm G(z,x,\cdot\,)(\lambda) = \frac{\phi_\pm(\lambda,x)}{\lambda-z}, \quad \lambda\in\sigma(T).
  \end{align}
 \end{lemma}
 
 \begin{proof}
  From~\eqref{eqnResolvent} we have   
  \begin{align*}
   \mathcal{F}_\pm G(z,x,\cdot\,)(\lambda) & = \int_\R G(z,x,s) \phi_\pm(\lambda,s) d\omega(s) = (T-z)^{-1} \phi_\pm(\lambda,\cdot\,)(x)  =  \frac{\phi_\pm(\lambda,x)}{\lambda-z}
  \end{align*}
  for every eigenvalue $\lambda\in\sigma(T)$.
 \end{proof}

 If $\tau$ is in the limit-circle case near $\pm\infty$, that is, when
\begin{align*}
 \int_\R \E^{\pm x}d|\omega|(x) < \infty,
\end{align*}
 then we are also able to determine the transforms of the Weyl solutions.

\begin{lemma}\label{lemTransWeyl}
 If $\tau$ is in the limit-circle case near $\pm\infty$, then for each $z\in\rho(T)$ 
 \begin{align}
  \mathcal{F}_\pm \phi_\mp(z,\cdot\,)(\lambda) = \frac{W(z)}{\lambda-z}, \quad \lambda\in\sigma(T).
 \end{align}
\end{lemma}

\begin{proof}
 From the Lagrange identity we have for every eigenvalue $\lambda\in\sigma(T)$ 
 \begin{align*}
  (\lambda-z)\int_\R \phi_\pm(\lambda,s) \phi_\mp(z,s)d\omega(s) & = \pm \lim_{s\rightarrow\pm\infty} \phi_\pm(\lambda,s) \phi_\mp'(z,s) - \phi_\pm'(\lambda,s) \phi_\mp(z,s).
 \end{align*}
 Now the claim follows since all functions $\phi_\pm(z,\cdot\,)$, $z\in\C$ have the same asymptotic behavior near $\pm\infty$.
\end{proof}

In the remaining part of this section we will solve the inverse problem for the class of weight measures in $\M$ which are sign definite. 
 The unique solvability in this case resembles the fact that solutions of the Camassa--Holm equation corresponding to measures which are of one sign exist for all times. In the indefinite case the inverse problem seems to be much more complicated and will not be discussed here. 
   
\begin{theorem}\label{thmInvExis}
 Let $\sigma\subseteq\R$ be a positive or negative discrete set satisfying \eqref{eqnEVsumabl} and for each $\lambda\in\sigma$ let $\gamma_{\lambda,\pm}^2\in\R$ such that $\lambda\gamma_{\lambda,\pm}^2>0$.
 Then the discrete measure  
   \begin{align}\label{eqnInvEx}
    \sum_{\lambda\in\sigma} \gamma_{\lambda,\pm}^{-2} \delta_\lambda
   \end{align}
  is the spectral measure of some unique (positive or negative) measure $\omega\in\M_\sigma$.
\end{theorem}

 \begin{proof}
  The results in \cite[Section~5]{bss} (or less immediately applicable also in \cite{dymmck}, \cite{kakr}, \cite{kotwat}) show that there are unique measures $\omega_{R,\pm}$, $R>0$ with the associated spectral measures 
  \begin{align*}
   \sum_{\lambda\in\sigma\cap[-R,R]} \gamma_{\lambda,\pm}^{-2} \delta_\lambda.
  \end{align*}
  From compactness of $\M_\sigma$ there is a subsequence $R_n$, $n\in\N$ such that $\omega_{R_n,\pm}$ converges to some $\omega\in\M$ in the weak$^\ast$ topology. Now the result in Theorem~\ref{thmweakstarconv} shows that~\eqref{eqnInvEx} is the spectral measure corresponding to $\omega$. 
 Also note that since the weight measure is uniquely determined by the spectral measure, this shows that the measures $\omega_{R,\pm}$ actually converge to $\omega$ in the weak$^\ast$ topology as $R\rightarrow\infty$.
 \end{proof}
 
 It is also possible to tell from the spectral measure whether the differential expression is in the limit-circle case near some endpoint or not. 

 \begin{corollary}\label{corinvproblc}
  For each positive or negative measure $\omega\in\M$ with corresponding spectrum $\sigma(T)$ and norming constants $\gamma_{\lambda,\pm}^2$, $\lambda\in\sigma(T)$ we have 
 \begin{align}\label{eqnLCPars}
  \int_\R \E^{\pm x} d\omega(x) = \sum_{\lambda\in\sigma(T)} \frac{ \gamma_{\lambda,\pm}^{-2}}{\lambda^{2}},
 \end{align}
 in the sense that one side is finite if and only if the other one is.
 \end{corollary}

\begin{proof}
 If the left-hand side in~\eqref{eqnLCPars} is finite, then equality follows from Proposition~\ref{propFtrans} and Lemma~\ref{lemTransWeyl}.
 Conversely, if the right-hand side is finite, then we may consider approximating measures $\omega_R$, $R>0$ as in the proof of Theorem~\ref{thmInvExis}. For each continuous cutoff function $\chi$ which takes values in $[0,1]$ we have
 \begin{align*}
 \int_\R \E^{\pm x} \chi(x)d\omega_R(x) \leq \int_\R \E^{\pm x}d\omega_R(x) = \mathop{\sum_{\lambda\in\sigma(T)}}_{|\lambda|\leq R} \frac{\gamma_{\lambda,\pm}^{-2}}{\lambda^2} \leq \sum_{\lambda\in\sigma(T)} \frac{\gamma_{\lambda,\pm}^{-2}}{\lambda^2}, \quad R>0.
  \end{align*}
  Because of $\omega_R\rightharpoonup^\ast\omega$ we infer that also  
  \begin{align*}
   \int_\R \E^{\pm x} \chi(x) d\omega(x) \leq \sum_{\lambda\in\sigma(T)} \frac{\gamma_{\lambda,\pm}^{-2}}{\lambda^2}
  \end{align*}
 and hence the left-hand side in~\eqref{eqnLCPars} is finite.
\end{proof}

In particular, this shows that finiteness of the left-hand side in~\eqref{eqnLCPars} is preserved under the Camassa--Holm flow (cf.\ \eqref{def:gamt} below).

\section{The Camassa--Holm flow}\label{secCH}

In this section we will apply our results to deduce some facts about the dispersionless Camassa--Holm equation
\begin{equation}
u_{t}-u_{txx}+3uu_x=2u_x u_{xx}+u u_{xxx},\qquad x,\,t\in\R.
\end{equation}
We say a family of finite signed measures $\omega(\,\cdot\,,t)$, parametrized by time $t\in I\subseteq\R$, evolves according to the Camassa--Holm flow if they are isospectral (with spectrum denoted by $\sigma$) and their associated norming constants evolve according to  
\begin{align}\label{def:gamt}
  \gamma_{\lambda,\pm}^2(t) = \E^{\mp\frac{t-t_0}{2\lambda}} \gamma_{\lambda,\pm}^2(t_0), \quad t,\, t_0\in I,~\lambda\in\sigma.
\end{align}
Provided $\om$ is sufficiently regular,
\begin{equation}
u(x,t) = \frac{1}{2} \int_\R \E^{-|x-s|}d\omega(s,t), \quad x\in\R,~t\in I
\end{equation}
will be a solution of the Camassa--Holm equation and in the general case we will use this as our definition of a weak
solution.

The most basic examples of (weak) solutions of the Camassa--Holm equation are so-called peakon solutions, given by  
 \begin{align*}
  \omega_{p}(\,\cdot\,,t) = 2c\,\delta_{ct-\eta}, \quad t \in\R
 \end{align*}
 for some nonzero constant $c\in\R$ and $\eta\in\R$.
 According to the preceding sections (in particular, see Proposition~\ref{propInverse}), the only eigenvalue $\lambda$ corresponding to these measures is given by $\lambda^{-1}=2c$.
 Moreover, it is immediate from~\eqref{eqnNorm} that the time dependent norming constant is given by 
 \begin{align*}
  \gamma_{\lambda,\pm}^2(t)
                             = \frac{1}{\lambda}\, \E^{\mp\frac{t}{2\lambda} \pm \eta},
 \end{align*}
 which shows that $\omega_p(\,\cdot\,,t)$, $t\in\R$ evolves according to the Camassa--Holm flow indeed.
 Note that the term peakon stems from the profile of the function  
 \begin{align*}
  \int_\R \E^{-|x-s|}d\omega_{p}(s,t) = \frac{1}{\lambda}\, \E^{-\left|x-\frac{t}{2\lambda}+\eta\right|}, \quad x,\, t\in\R,
 \end{align*}
 which will be re-encountered in the following section. 
 More generally, it is also possible to consider solutions which at each fixed time $t\in I$ are a finite sum of weighted Dirac measures, referred to as multi-peakon solutions. 
 However, in this case the height and position of the single peaks evolve in a nonlinear fashion.  

 In order to state the following result, let $\omega_n(\,\cdot\,,t)$, $n\in\N$ and $\omega(\,\cdot\,,t)$ be some families of finite signed measures parametrized by time $t\in I$, evolving according to the Camassa--Holm flow with spectra contained in some discrete set $\sigma$ satisfying~\eqref{eqnEVsumabl}. 

\begin{theorem}
 If the measures $\omega_n(\,\cdot\,,t_0)$ converge to $\omega(\,\cdot\,,t_0)$ in the weak$^\ast$ topology for some $t_0\in I$, then for each fixed $t\in I$  the measures $\omega_n(\,\cdot\,,t)$ converge to $\omega(\,\cdot\,,t)$ in the weak$^\ast$ topology provided they are uniformly bounded. 
\end{theorem}
 
 \begin{proof}
 By assumption and Theorem~\ref{thmweakstarconv} there is a subsequence $\omega_{n_k}(\,\cdot\,,t_0)$ such that $\sigma(S_{n_k}(t_0))$ converges to some $\sigma_\infty\subseteq\sigma$ as $k\rightarrow\infty$ and 
   \begin{align*}
   \lim_{k\rightarrow\infty} \lambda \gamma_{n_k,\lambda,\pm}^2(t_0) = \begin{cases} 0, & \lambda\in\sigma_\pm, \\ \lambda \gamma_{\lambda,\pm}^2(t_0) \prod_{\kappa\in\sigma_\pm} \left(1-\frac{\lambda}{\kappa}\right)^2, & \lambda\in\sigma(S(t_0)), \\ \infty, & \lambda\in\sigma_\mp,  \end{cases}
  \end{align*}
  for some partition $\sigma_-\dot{\cup}\,\sigma_+$ of $\sigma_\infty\backslash\sigma(S(t_0))$.
 But then for each fixed $t\in I$, $\sigma(S_{n_k}(t))$ converges to $\sigma_\infty$ as well and 
    \begin{align*}
   \lim_{k\rightarrow\infty} \lambda \gamma_{n_k,\lambda,\pm}^2(t) = \begin{cases} 0, &  \lambda\in\sigma_\pm, \\ \lambda \gamma_{\lambda,\pm}^2(t) \prod_{\kappa\in\sigma_\pm} \left(1-\frac{\lambda}{\kappa}\right)^2, & \lambda\in\sigma(S(t)), \\ \infty, & \lambda\in\sigma_\mp.  \end{cases}
  \end{align*}
  Hence Corollary~\ref{corInvCont} and our inverse uniqueness theorem imply that $\omega_{n_k}(\,\cdot\,,t)$ converges to $\omega(\,\cdot\,,t)$ in the weak$^\ast$ topology. 
  Now the claim follows from these considerations upon using a simple compactness argument.
 \end{proof}

 As a simple consequence of Theorem~\ref{thmInvExis} we may immediately deduce existence of global solutions of the Camassa--Holm equation for initial data which are of one sign. 
 More precisely, Theorem~\ref{thmInvExis} guarantees that for each given positive or negative weight measure $\omega_0\in\M$ there is a unique family of finite signed measures $\omega(\,\cdot\,,t)$, $t\in\R$ evolving according to the Camassa--Holm flow, with the initial condition $\omega(\,\cdot\,,0)=\omega_0$. 
 The following result of the type of \cite[Theorem~3.1]{holray} shows that this solution may be approximated by multi-peakon solutions in the weak$^\ast$ topology, locally uniformly in time.
  Here by locally uniformly in time we mean with respect to some (arbitrary) metric which induces the weak$^\ast$ topology on $\M_\sigma$, where $\sigma$ is the spectrum associated with $\omega_0$. 
 In order to state this result, consider for each $n\in\N$ the multi-peakon solution $\omega_n(\,\cdot\,,t)$, $t\in\R$ which is obtained from Theorem~\ref{thmInvExis} upon cutting off the spectral measures of $\omega(\,\cdot\,,t)$, $t\in\R$ outside of the interval $[-n,n]$. 
 
 \begin{corollary}
  The measures $\omega_n(\,\cdot\,,t)$ of the multi-peakon solutions converge to $\omega(\,\cdot\,,t)$ in the weak$^\ast$ topology, locally uniformly in $t\in\R$. 
 \end{corollary}
 
 \begin{proof}
  For each $T>0$ the set
 \begin{align*}
    \M_{\sigma}(T) = \lbrace \omega\in\M_\sigma \,|\, \pm \gamma_{\omega,\lambda,\pm}^{2} \in [\pm \gamma_{\lambda,\pm}^{2}(T), \pm \gamma_{\lambda,\pm}^{2}(-T)], ~ \lambda\in\sigma(S_\omega)\rbrace
 \end{align*}
 is compact with respect to the weak$^\ast$ topology by Theorem~\ref{thmweakstarconv}. 
 Moreover, the mapping $\omega\mapsto \rho$ is uniformly continuous on $\M_{\sigma}(T)$, where the set of spectral measures $\rho$ (corresponding to $\M_\sigma$) is identified with the sequences $(\rho(\lbrace\lambda\rbrace))_{\lambda\in\sigma}$ and equipped with the product topology (pointwise convergence) there. 
 Hence, the inverse of this mapping is uniformly continuous as well which proves the claim. 
 \end{proof}
   
 Another application of the results in this paper concern the construction of a Lipschitz metric for the Camassa--Holm equation, as done recently in \cite{grhora}. 
 More precisely, our inverse uniqueness result allows us to define such a metric on $\M$ using the corresponding spectral data.  
 Therefore consider some bounded metric $d$ on the projective extended real line $\R^\infty = \R\cup\lbrace\infty\rbrace$ which obeys
\begin{align*}
 d(\tau_1 + t, \tau_2 + t) \leq \E^t d(\tau_1, \tau_2), \quad \tau_1,\, \tau_2\in\R^\infty,
\end{align*}
for each positive $t\geq 0$ and
\begin{align*}
 d(\tau + t, \tau + s) \leq |t-s|, \quad s,\, t\in\R
\end{align*}
for every $\tau\in\R$. 
 For example, one may take the metric given by 
  \begin{align*}
   d(\tau_1, \tau_2) = \min\left(\int_{\tau_1}^{\tau_2} \frac{d\eta}{1+\eta^2}, \int_{-\infty}^{\tau_1} \frac{d\eta}{1+\eta^2} + \int_{\tau_2}^\infty \frac{d\eta}{1+\eta^2} \right)
\end{align*}
 for all $\tau_1$, $\tau_2\in\R^\infty$ with $\tau_1\leq \tau_2$. 
 Given such a metric $d$, we are able to define a metric on $\M_\sigma$ (where $\sigma\subseteq\R$ satisfies \eqref{eqnEVsumabl}) by   
\begin{align*}
   d_{\sigma,\pm}(\omega_1,\omega_2) = \sum_{\lambda\in\sigma} \frac{1}{|\lambda|} d\left(\mp\sgn(\lambda) \ln \left(\lambda \gamma_{1,\lambda,\pm}^{2}\right), \mp \sgn(\lambda) \ln \left(\lambda \gamma_{2,\lambda,\pm}^{2}\right)\right)
\end{align*}
 for $\omega_1$, $\omega_2\in\M_\sigma$, 
 where we make the convention that $\ln(\lambda \gamma_{j,\lambda,\pm}^2) = \infty$ if $\lambda$ is not an eigenvalue of $S_j$, in formal agreement with~\eqref{eqnNorm}. 
 This metric is easily verified to be a Lipschitz metric for the Camassa--Holm equation in the following sense. 
 If $\omega_1(\,\cdot\,,t)$, $\omega_2(\,\cdot\,,t)\in\M_\sigma$, $t\in I$ evolve according to the Camassa--Holm flow, then 
 \begin{align}\label{eqnLM1}
   d_{\sigma,\pm}(\omega_1(\,\cdot\,,t),\omega_2(\,\cdot\,,t)) \leq \E^{\frac{t-t_0}{2\Lambda}} d_{\sigma,\pm}(\omega_1(\,\cdot\,,t_0),\omega_2(\,\cdot\,,t_0)), \quad t,\, t_0\in I,~ t\geq t_0.
  \end{align} 
 Here, $\Lambda$ is some lower bound for the absolute values of the associated eigenvalues, for example 
 \[
  \Lambda= \min_{\lambda\in\sigma} |\lam| >0.
 \]
 Moreover, if $\omega(\,\cdot\,,t)\in\M_\sigma$, $t\in I$ evolves according to the Camassa--Holm flow, then
 \begin{align}\label{eqnLM2}
  d_{\sigma,\pm}(\omega(\,\cdot\,,t),\omega(\,\cdot\,,s)) \le \frac{|t-s|}{2} \sum_{\lambda\in\sigma} \frac{1}{|\lambda|^2}, \quad t,\, s\in I.
 \end{align}
It is also possible to extend the metric $d_{\sigma,\pm}$ to all of $\M$ by 
  \begin{align*}
    d_\pm(\omega_1,\omega_2) = d_{\sigma(S_1)\cup\sigma(S_2),\pm}(\omega_1,\omega_2)
  \end{align*}
  for $\omega_1$, $\omega_2\in\M$, 
  which is a Lipschitz metric as well, in the sense that \eqref{eqnLM1} and \eqref{eqnLM2} hold but with the Lipschitz constant depending on the (associated spectra of the) respective solutions of the Camassa--Holm equation in an obvious way. Theorem~\ref{thmweakstarconv} and Corollary~\ref{corInvCont} show that the topology induced by this metric $d_\pm$ is not comparable with the weak$^\ast$ topology on $\M$. However, restricted to isospectral sets, the two topologies coincide.

 \section{Long-time asymptotics}\label{secLT}
 
In this final section we will prove that solutions of the Camassa--Holm equation asymptotically split into an (in general infinite) train of well separated single peakons, each of which corresponding to an eigenvalue of the underlying isospectral problem. 
Therefore, consider a family of finite signed measures $\omega(\,\cdot\,,t)$, $t\in I\subseteq\R$ which evolve according to the Camassa--Holm flow with spectrum
denoted by $\sigma$.
The continuity results for the inverse spectral problem in Section~\ref{secCC} now yield the following long-time asymptotics. 

 \begin{theorem}\label{thmLT}
  Suppose that the measures $\omega(\,\cdot\,,t)$ are uniformly bounded in $t\in I$. 
  Then for each $x_0\in\R$ and $c\in\R$ we have the asymptotics
 \begin{align}\label{eqnLTA}
   \int_\R \E^{-|x-s|}d\omega(s,t) = \sum_{\lambda\in\sigma} \frac{1}{\lambda}\, \E^{-\left|x-\frac{t}{2\lambda} + \eta_\lambda\right|} + \oo(1)
 \end{align}
 as $t\rightarrow\infty$ in $I$ along the ray $x=x_0 + ct$, where the phase shifts $\eta_\lambda$ are given by
 \begin{align}
  \E^{\pm\eta_\lambda} = \lambda \gamma_{\lambda,\pm}^2(t_0) \E^{\pm\frac{t_0}{2\lambda}} \mathop{\prod_{\kappa\in\sigma}}_{\pm\kappa^{-1}>\pm\lambda^{-1}} \biggl(1-\frac{\lambda}{\kappa}\biggr)^{-2}, \quad \lambda\in\sigma, ~t_0\in I.
 \end{align}
\end{theorem}

\begin{proof}
 For every $t\in I$ let $\omega_t$ be the measure given by $\omega_t(B) = \omega(B + x_0 + ct,t)$ for each Borel set $B\subseteq\R$. 
 Since translations obviously leave the spectrum invariant, $\omega_t$ lies in $\M_\sigma$ for each $t\in I$ and we furthermore have
 \begin{align}\tag{$*$}\label{eqnuTrans} 
  \int_\R \E^{-|x_0 + ct-s|}d\omega(s,t) = \int_\R \E^{-|s|} d\omega_t(s), \quad t\in I.
 \end{align}
 Moreover, the solutions $\phi_{t,\pm}$ associated with $\omega_t$ are simply given by
 \begin{align*}
  \phi_{t,\pm}(z,x) = \E^{\pm\frac{x_0 + ct}{2}} \phi_\pm(z,x+x_0+ct,t), \quad x\in\R,~z\in\C,~t\in I
 \end{align*}
 and hence the norming constants $\gamma_{t,\lambda,\pm}^2$, $\lambda\in\sigma$ corresponding to $\omega_t$ are given by
 \begin{align*}
  \gamma_{t,\lambda,\pm}^2 = \gamma_{\lambda,\pm}^2(t) \E^{\pm (x_0 + ct)} =  \gamma_{\lambda,\pm}^2(t_0)\E^{\pm\left(x_0 + \frac{t_0}{2\lambda}\right)} \E^{\pm t\left(c-\frac{1}{2\lambda}\right)}, \quad \lambda\in\sigma,~ t\in I.
 \end{align*}
 Now since these quantities converge (including possibly to infinity) as $t\rightarrow\infty$ in $I$ we infer from Corollary~\ref{corInvCont} that $\omega_t\rightharpoonup^\ast\omega_c$ for some $\omega_c\in\M_\sigma$ whose corresponding spectrum contains at most one eigenvalue.
 
 We will distinguish between the two possible cases. 
 First, if there is some $\lambda_c\in\sigma$ with $\lambda_c^{-1}=2c$, then the only eigenvalue associated with $\omega_c$ is $\lambda_c$ with corresponding norming constant given by
 \begin{align*}
  \gamma_{\lambda_c,\pm}^2(t_0) \E^{\pm\left(x_0 + ct_0\right)} \mathop{\prod_{\kappa\in\sigma}}_{\pm\kappa^{-1}> \pm\lambda_c^{-1}} \left(1-\frac{\lambda_c}{\kappa}\right)^{-2}
 \end{align*}
 in view of Corollary~\ref{corInvCont}.
 Hence a comparison with the peakon solution in Section~\ref{secCH} 
 shows that $\omega_c = 2 c \delta_{x_c}$ with $x_c = x_0 + \eta_{\lambda_c}$.
 In view of~\eqref{eqnuTrans} this yields the asymptotics 
 \begin{align*}
  \int_\R \E^{-|x-s|}d\omega(s,t) = \frac{1}{\lambda_c} \, \E^{-\left| x -\frac{t}{2\lambda_c} +\eta_{\lambda_c}\right|} + \oo(1),
 \end{align*}
 as $t\rightarrow\infty$ in $I$ along the ray $x=x_0 + ct$. 
 Secondly, if the spectrum $\sigma_\infty$ is empty, then we infer that $\omega_t\rightharpoonup^\ast 0$ and~\eqref{eqnuTrans} shows that the left-hand side of~\eqref{eqnLTA} is of order $\oo(1)$ as $t\rightarrow\infty$ in $I$ along the ray $x=x_0 + ct$.  
 In order to finish the proof (in both cases) note that all remaining terms in the sum on the right-hand side of~\eqref{eqnLTA} are of order $\oo(1)$ as $t\rightarrow\infty$ in $I$ along the ray $x=x_0+ct$. 
 A simple application of the dominated convergence theorem shows that the sum over all these remaining terms is of order $\oo(1)$ as well. 
\end{proof} 

 Of course there are no essential differences when considering the asymptotics for $t\rightarrow -\infty$ in $I$. 
  In fact, if some family of finite signed measures $\omega(\,\cdot\,,t)$, $t\in I$ evolves according to the Camassa--Holm flow, then so does $\omega(\,-\,\cdot\,,-t)$, $t\in -I$. 
  Hence, asymptotics for $t\rightarrow -\infty$ in $I$ may be deduced immediately from Theorem~\ref{thmLT}. 


\end{document}